\documentclass[11pt]{amsart}

\usepackage{amssymb,latexsym}
\usepackage{amsthm}
\usepackage{amsmath}
\usepackage{amsfonts}
\usepackage{graphicx}
\usepackage[all]{xy}
\usepackage{fancyhdr}
\usepackage[top=1in,bottom=1in,left=1.25in,right=1.25in]{geometry}

\newtheorem{theorem}{Theorem}[section]
\newtheorem{definition}[theorem]{Definition}

\newtheorem{proposition}[theorem]{Proposition}
\newtheorem{corollary}[theorem]{Corollary}
\newtheorem{lemma}[theorem]{Lemma}

\def\Q{\mathbb{Q}}
\def\F{\mathbb{F}}
\def\R{\mathbb{R}}
\def\Z{\mathbb{Z}}
\def\A{\mathbb{A}}

\def\C{\mathbb{C}}
\def\G{\mathbb{G}}

\def\M{\mathcal{M}}

\def\vp{\varpi}

\def\lan{\langle}
\def\ran{\rangle}

\def\lra{\longrightarrow}
\def\ra{\rightarrow}
\def\ov{\overline}
\def\ul{\underline}
\def\wh{\widehat}
\def\wt{\widetilde}
\def\st{\stackrel}
\def\tr{\textrm}

\usepackage[pdftex]{hyperref}

\begin{document}
\title{On the $\ell$-adic cohomology of some $p$-adically uniformized Shimura varieties}
\author{Xu Shen}
\date{}
\address{Fakult\"at f\"{u}r Mathematik\\
Universit\"at Regensburg\\
Universitaetsstr. 31\\
93040, Regensburg, Germany} \email{xu.shen@mathematik.uni-regensburg.de}

\address{Current address: Morningside Center of Mathematics\\
	No. 55, Zhongguancun East Road\\
	Beijing 100190, China}\email{shen@math.ac.cn}

\renewcommand\thefootnote{}

\begin{abstract}
We determine the Galois representations inside the $\ell$-adic cohomology of some unitary Shimura varieties at split places where they admit uniformization by finite products of Drinfeld upper half spaces. Our main results confirm Langlands-Kottwitz's description of the cohomology of Shimura varieties in new cases.
\end{abstract}

\maketitle
\tableofcontents
\section{Introduction}
The aim of this article is to determine the Galois representations inside the $\ell$-adic cohomology of some unitary Shimura varieties at split places where they admit uniformization by finite products of Drinfeld upper half spaces (\cite{RZ} theorem 6.50 and \cite{Va2}). The main results confirm Langlands-Kottwitz's description of the cohomology of Shimura varieties in new cases.

For the Shimura varieties with good reductions at $p$ ($\neq \ell$), Langlands and Kottwitz have given a conjectural description of the Galois representations inside the cohomology, cf. \cite{Ko4}. Roughly it says that, the Galois representation associated to an automorphic representation when restricting to a place above $p$ is given by the local Langlands correspondence for the local reductive group. To prove such a result, Langlands's idea is to analyze the cohomology of Shimura varieties by computing the alternating sum of the traces of Hecke operators twisted by a Frobenius correspondence on the cohomology. By Lefschetz trace formula, this needs to understand the set of points on Shimura varieties over finite fields. Kottwitz introduced some group theoretic triples $(\gamma_0;\gamma,\delta)$ in \cite{Ko2} which, roughly speaking, parameterize isogeny classes of points of Shimura varieties over finite fields. There he also proved a formula for the traces of Hecke operators twisted by a Frobenius by using these triples. Then to get the desired description, one should stabilize this formula and compare it with the Arthur-Selberg trace formula, see \cite{Ko4}. In some cases this description has been proved, for example see \cite{Ko3, Mo}. Here in our case, even for the maximal level at $p$, these varieties have bad reductions. And in fact we also want to treat the cases of arbitrary levels at $p$.

In \cite{Sch3} Scholze has developed the Langlands-Kottwitz approach for some PEL Shimura varieties with arbitrary level at $p$. There the local hypothesis are made to ensure the local reductive groups are quasi-split, so that there is no problem for the definition of Kottwitz triples. Also, Scholze proved that the set of all equivalent effective Kottwitz triples can parameterize all the fixed points. The key new ingredient is to define some test functions by deformation spaces of $p$-divisible groups with some additional structures. This approach avoids the study of local models of Shimura varieties with bad reductions, and gives less information about these functions. However, the definition of these test functions is conceptually elegant, and sufficient for applications in many ways. Scholze then studied the properties of these functions and proved a formula similar to the one of Kottwitz in \cite{Ko2}. Using this formula Scholze and Shin in \cite{SS} have proved some character identities about the transfers of the test functions defined in \cite{Sch3}, and deduce many results about the cohomology of some compact unitary Shimura varieties for arbitrary level at $p$, which confirm the expected descriptions of Langlands-Kottwitz. Note to have such a description one needs to know the local Langlands correspondence for the related reductive groups. In \cite{SS} their assumptions are made such that the local reductive groups at $p$ are products of Weil restriction of general linear groups.

For the PEL Shimura varieties with reductive groups $G$ non quasi-split at $p$, we also want to describe their points modulo $p$ and their $\ell$-adic cohomology. However, there are some group theoretic problems due to the non quasi-split property. Namely, in this set up the set of equivalent Kottwitz triples (in the usual sense) will not be enough to parameterize all the points valued in a finite field $\F_{p^t}$. We can indeed find the pairs $(\gamma,\delta)$ associated to each isogeny class over $\F_{p^t}$. But the conjugacy class of the norm $N\delta$ does not always contain an element of $G(\Q_p)$, cf. \cite{Ko6}. This is an obstruction to find a $\gamma_0\in G(\Q)$ such that $(\gamma_0;\gamma,\delta)$ forms a Kottwitz triple. Nevertheless, we can introduce some reasonable test functions $\phi_{\tau,h}$ at $p$ in the same way as \cite{Sch3} whose twisted orbital integrals contribute to the trace formula. This was already noted by Scholze in \cite{Sch3}. Following \cite{Ra} (conjecture 5.7) and \cite{Ra1} (conjecture 10.2), one conjectures that if the conjugacy class of the norm $N\delta$ does not contain an element of $G(\Q_p)$, then the twisted orbital integral vanishes
\[TO_{\delta\sigma}(\phi_{\tau,h})=0.\]This is certainly a new phenomenon in the non quasi-split case.
If one can prove this result, then only the points parameterized by the Kottwitz triples contribute to the trace formula, and one has a similar formula as \cite{Sch3} and \cite{Ko4}. Going through further, one can continue the process of stabilization or pseudo-stabilization to compute the cohomology. To get the desired description of the cohomology, one still needs to know the local Langlands correspondence for $G(\Q_p)$ and a suitable character identity for the twisted transfer of $\phi_{\tau,h}$ (cf. \cite{SS} conjecture 7.1, which is a weaker form of \cite{Hai2} conjectures 6.1.1 and 6.2.3).

This paper deals with a special example where we can prove the above two main points for the test functions $\phi_{\tau,h}$. Also, for this case at hand, the local Langlands correspondences has been known. So we can get desired description of the cohomology. Let $Sh_K$ be a Shimua variety over its local reflex field $E$ with the open compact subgroup in the form $K=K^pK_p\subset G(\A_f)$, such that at $p$ it admits uniformization by $r$ products of Drinfeld moduli spaces $\M_{Dr, K_p}$ with level $K_p$. Let $\xi$ be an irreducible algebraic representation of $G$ over $\ov{\Q}_\ell$. By a standard construction, we have a $\ov{\Q}_\ell$-local system $\mathcal{L}_\xi$ on each Shimura variety $Sh_K$ for $K\subset G(\A_f)$. We are interested in the virtual $G(\A_f)\times W_E$-representation defined by the alternating sum of $\ell$-adic cohomology
\[H_\xi=\sum_{i}(-1)^i\varinjlim_{K}H^i(Sh_K\times\ov{\Q}_p, \mathcal{L}_\xi),\]
where $K$ runs over the set of open compact subgroups of $G(\A_f)$. The local reductive group $G_{\Q_p}$ is a product of some inner forms of $GL_n$ (together with $\G_m$), so the local Langlands correspondence has been known, cf. \cite{ABPS, HS}. For any smooth irreducible representation $\pi_p$ of $G(\Q_p)$, let $\varphi_{\pi_p}$ be the associated local Langlands parameter. Recall associated to the Shimura data $(G,h^{-1})$ we have the representation $r_{-\mu}$ of the Langlands dual group $^L(G_{E})$ introduced in \cite{Ko1} lemma 2.1.2. The main theorem is as follows.
\begin{theorem}
With the notations as above, we have an identity
\[H_\xi=\sum_{\pi_f}a(\pi_f)\pi_f\otimes(r_{-\mu}\circ\varphi_{\pi_p}|_{W_E})|-|^{r(1-n)/2}\]
as virtual $G(\Z_p)\times G(\A_f^p)\times W_E$-representations. Here $\pi_f$ runs through irreducible admissible representations of $G(\A_f)$, the integer $a(\pi_f)$ is as in \cite{Ko3} p. 657.
\end{theorem}

We first prove the theorem for case $r=1$ by using the results of Boyer \cite{Bo} and Dat \cite{Dat}, thus avoiding the counting points method for these Shimura varieties. In fact, the global $p$-adic uniformization leads to a Hochschild-Serre spectral sequence for the $\ell$-adic cohomology of these Shimura varieties. From which we get a formula (in \cite{F} Fargues called it as a $p$-adic Matsushima formula) for $H_\xi$. Then using Boyer's description of the cohomology of Lubin-Tate spaces, the Faltings-Fargues isomorphism for the towers of Lubin-Tate and Drinfeld, and Dat's results about extension of elliptic representations, we get the desired formula in the theorem. Here we prove the identity as $G(\A_f)\times W_E$-representations. We note that essentially the same idea had already appeared in \cite{Dat} section 5.

To prove the result in general case, we develop the theory of test functions by means of deformation spaces of $p$-divisible groups in our context as in \cite{Sch3}. We also adapt some notations from \cite{Sch2}. To prove the vanishing results and character identities for these test functions, we use the formula for $H_\xi$ in the case $r=1$ proved previously. Thus these are some global arguments. We note that in \cite{Ra} Rapoport conjectured the vanishing results for the example there with maximal level at $p$, and by some explicit combinatorial description of the test function Waldspurger proved this conjecture in the case $r=1$. There the test functions were constructed by the theory of local models which describes the bad reduction of the Shimura varieties. These functions satisfy the required character identities. So in this case $r=1$ and the level structure $K_p$ at $p$ is maximal, the result was more or less known in \cite{Ra}. In fact in loc. cit. Rapoport restricted on the trivial coefficients and concentrated on the local semisimple zeta functions (see the below corollary).

In \cite{Sh} we shall use our results for the test functions to describe the cohomology of quaternionic and related unitary Shimura varieties at ramified places, see section 7 for more details.

In a series of papers \cite{H1, H2, H3}, Harris had studied the supercuspidal part of the cohomology. At that time, one did not know the local Galois representations inside the cohomology are given by the local Langlands correspondence, except for the cases $n<p$ (\cite{H2}). In fact, in these papers Harris tried to prove the local Langlands correspondence for $GL_n$ by studying the supercuspidal part of the cohomology of these Shimura varieties. Later, in \cite{HT} Harris and Taylor successfully achieved this by studying the cohomology of another class of Shimura varieties. We note that the same result as in the above theorem for Harris-Taylor's Shimura varieties was proved in \cite{Sch1} implicitly.

From this theorem we get the following corollary concerning the local semisimple zeta functions of our Shimura varieties. Let $\wt{E}$ be the global reflex field, and $\nu$ be a place of $\wt{E}$ above $p$ such that $E=\wt{E}_\nu$.
\begin{corollary}
In the situation of the theorem, let $K\subset G(\A_f)$ be any sufficiently small open compact subgroup. Then the semisimple local Hasse-Weil zeta function of $Sh_K$ at the place $\nu$ of $\wt{E}$ is given by
\[\zeta_{\nu}^{ss}(Sh_K,s)=\prod_{\pi_f}L^{ss}(s-r(n-1)/2,\pi_p,r_{-\mu})^{a(\pi_f)dim\pi_f^K}.\]
\end{corollary}
In the case $r=1$ Dat has proved the Weight-Monodromy conjecture for these Shimura varieties, cf. \cite{Dat} 5.2. Then by \cite{Ra} section 2 one can recover the classical Hasse-Weil zeta function.
\begin{corollary}
Let $r=1$ and $K\subset G(\A_f)$ be any sufficiently small open compact subgroup in the situation of the theorem. Then the local Hasse-Weil zeta function of $Sh_K$ at the place $\nu$ of $\wt{E}$ is given by
\[\zeta_{\nu}(Sh_K,s)=\prod_{\pi_f}L(s-r(n-1)/2,\pi_p,r_{-\mu})^{a(\pi_f)dim\pi_f^K}.\]
\end{corollary}

In \cite{Ito} Ito proved the Weight-Monodromy conjecture for the varieties which are $p$-adically uniformized by the maximal level Drinfeld upper half spaces. In theorem 6.2 of loc. cit. an application to the local zeta function was also presented. These varieties are the Galois twisted versions of the connected components of our Shimura varieties studied here with $r=1$ and $K_p$ maximal, see theorem 6.50 of \cite{RZ} and 2.13 of \cite{Va2}.

Shortly after the first version of this paper, Mieda claimed that the Weight-Monodromy conjecture holds true for the general $p$-adically uniformized Shimura varieties studied here, i.e. $r$ is not necessary 1, and one can argue as in section 3 to prove the above theorem 1.1, cf. \cite{Mie}. In particular in corollary 1.3 $r$ can be an arbitrary positive integer. However, as the reader can see in this introduction, our general purpose is trying to prove results for as many cases as possible. Our theory of test functions developed in sections 4-6 by Scholze's method, will be used in \cite{Sh} in an essential way to prove results for some other Shimura varieties.

We give an overview of the content of this article. In section 2 we introduce the $p$-adically uniformized Shimura varieties which we are interested in. In section 3 we deduce the cohomology of the Shimura varieties uniformized by one Drinfeld moduli spaces by using the results of Boyer and Dat. Then in section 4 we define the test functions $\phi_{\tau,h}$ by Scholze's method, and list their properties. Here we just state the facts since the arguments and proofs are the same as those in \cite{Sch3}. In section 5 we prove the vanishing property of these test functions by global method, and establish the trace formula as a sum over the set of equivalent Kottwitz triples. In section 6 we use the formula for $H_\xi$ in the case $r=1$ proved in section 3 to deduce the character identity of the transfers $f_{\tau, h}$ of $\phi_{\tau,h}$. Finally in section 7 we deduce the theorem for the general case. Corollaries for the local (semisimple) zeta functions of Shimura varieties are stated.\\
 \\
\textbf{Acknowledgments.} I would like to thank Peter Scholze sincerely, since without his encouragement and suggestions this work would not be finished. I want to thank Yoichi Mieda and Michael Rapoport for their useful remarks. I should also thank the referee for careful reading and valuable suggestions. This work started while the author was a postdoc at the Mathematical Institute of the University of Bonn, and was finished after the author moved to the University of Regensburg. This work was supported by the SFB/TR 45 ``Periods, Moduli Spaces and Arithmetic of Algebraic Varieties'' and the SFB 1085 ``Higher Invariants'' of the DFG.

\section{Some $p$-adically uniformized Shimura varieties}
We now introduce some unitary Shimura varieties which admit $p$-adic uniformization by finite products of Drinfeld upper half spaces. They were first introduced by Rapoport-Zink \cite{RZ} and Varshavsky \cite{Va1, Va2} as higher dimensional generalization of the Cherednik's theorem for Shimura curves as presented in \cite{BC} and \cite{D}. We note that some special higher dimensional cases already appeared in \cite{Ra}.

Let $p$ be a prime number. Fix an imaginary quadratic field $K$ in which $p$ splits. The two primes of $K$ above $p$ will be denoted by $u$ and $u^c$, and the complex conjugation of $Gal(K/\Q)$ will be denoted by $c$. Let $F^+|\Q$ be a totally real field of degree $N$. Set $F=F^+K$, so that $F$ is a CM-field with maximal totally real subfield $F^+$. Let $\varpi_1, \vp_2,\dots, \vp_s$ denote the primes of $F$ above $u$, and let $\nu_1,\nu_2,\dots,\nu_s$ denote their restrictions to $F^+$. Fix an integer $1\leq r\leq s$. Let $B/F$ denote a central division algebra of dimension $n^2$ over $F$ such that
\begin{itemize}
\item the opposite algebra $B^{op}$ is isomorphic to $B\otimes_{K,c}K$;
\item at any place $x$ of $F$ which is not split over $F^+$, $B_x$ is split (here and in the following $B_x=B\otimes F_x$);
\item at the places $\vp_1,\dots,\vp_r, \vp_1^c,\dots,\vp_r^c$ (for $1\leq i\leq s, \vp_i^c$ is the place over $u^c$ which induces also $\nu_i$ on $F^+$) $B$ is ramified with invariants
    \[inv_{\vp_i}B=\frac{1}{n}, \, inv_{\vp_i^c}B=-\frac{1}{n};\]
\item at the places $\vp_{r+1},\dots,\vp_s,\vp_{r+1}^c,\dots,\vp_s^c$ the invariants of $B$ are arbitrary but satisfy \[inv_{\vp_i}B=-inv_{\vp_i^c}B.\]
\end{itemize}
Note that we assume in particular $s=t$ by the notation of \cite{RZ} 6.38.

We assume that there is an involution of second kind $\ast$ on $B$. Moreover, we can choose some alternating pairing $\lan,\ran$ on $V\times V\ra \Q$ for the $B\otimes_FB^{op}$ module $V:=B$, which corresponds to another involution of second kind $\sharp$ on $B$. The associated reductive group $G/\Q$ is defined by
\[G(R)=\{(g,\lambda)\in (B^{op}\otimes_\Q R)^\times\times R^\times|\,gg^\sharp=\lambda\},\]
for any $\Q$-algebra $R$. Let $G_1$ be the kernel of the map $G\ra \mathbb{G}_m, (g,\lambda)\mapsto \lambda$, which can be viewed as a group over $F^+$. For $1\leq i\leq r$,  choose distinguished embeddings $\tau_i: F^+\hookrightarrow \R$. As in \cite{RZ} 6.40, we assume that we can make the choice of the alternating pairing on $V\times V\ra \Q$ and the isomorphism $\iota: \ov{\Q}_p\simeq \C$ such that
\begin{itemize}
\item if $\sigma: F^+\hookrightarrow \R$ is an embedding, then $G_1\times_{F^+,\sigma}\R$ is isomorphic to the unitary group $U(1,n-1)$ if $\sigma=\tau_i$ for $1\leq i\leq r$ and $U(n)$ otherwise;
\item under the bijection $Hom(F^+,\C)=Hom(F^+,\R)\ra Hom(F^+,\ov{\Q}_p)$ induced by $\iota$, $\tau_1,\dots,\tau_r$ induce the primes $\nu_1,\dots,\nu_r$ of $F^+$ above $p$.
\end{itemize}
Let $B_{\vp_i}^\times$ be the local reductive group over $F_{\vp_i}$ associated to the units in $B_{\vp_i}$ for $1\leq i\leq r$. Then under our assumptions we have
\[G_{\Q_p}\simeq Res_{F_{\vp_1}}B_{\vp_1}^\times \times\cdots\times Res_{F_{\vp_r}}B_{\vp_r}^\times \times G_{D'}\times \G_m\]with obvious definition of the factor $G_{D'}$ which is associated to the semisimple algebra $D'=\prod_{i=r+1}^sB_{\vp_i}$ over $\Q_p$. Let $E$ be the composition of the fields $F_{\vp_1},\dots,F_{\vp_r}$, which will be our local reflex field.

As in \cite{RZ} 6.37 we have a homomorphism $h: Res_{\C|\R}\G_m\ra G_\R$ which corresponds to our signature condition. Then the datum $(G, h^{-1})$ defines a projective Shimura variety $Sh_K$ over $E$ for any compact open subgroup $K\subset G(\A_f)$, cf. \cite{Ko2}. The conjugacy class of the cocharacter $\mu: \G_{m}\lra G_{\ov{\Q}_p}$ associated to $h$ is defined over $E$, cf. \cite{RZ} 6.40. For sufficiently small open compact subgroup $K^p\subset G(\A_f^p)$, we have a projective scheme $S_{K^p}$ over $O_E$ (the integer ring of $E$) which is a moduli space of some abelian varieties with additional structures. This moduli space is defined in a similar way as those introduced in \cite{Ko2}, but contrary to the later case, it is not smooth. For any locally noetherian scheme $S$ over $O_E$, $S_{K^p}(S)$ is the isomorphism classes of quadruples $(A,\lambda,\iota,\ov{\eta}^p)$ consisting of
\begin{itemize}
\item a projective abelian scheme $A$ over $S$ up to prime-to-$p$ isogeny,
\item a polarization $\lambda: A\ra A^D$ of degree prime to $p$ (here and in the following, the upper subscript $D$ means the Cartier dual),
\item a homomorphism $\iota: O_B\ra End(A)$ satisfying the determinant condition and compatible with $\lambda$,
\item a level structure $\ov{\eta}^p$ of type $K^p$.
\end{itemize} For more details we refer to \cite{RZ} definition 6.9. As usual on generic fibers we have the isomorphism
\[S_{K^p}\times_{O_E}E\simeq \coprod_{ker^1(\Q,G)}Sh_{K_p^0K^p},\]where $K_p^0\subset G(\Q_p)$ is the maximal open compact subgroup $O_{B_{\vp_1}}^\times\times\cdots\times O_{B_{\vp_r}}^\times\times O_{D'}^\times\times\Z_p^\times$.

We set $D_i:=B_{\vp_i}$ from now on. For any locally noetherian scheme $S$ over $O_E$ on which $p$ is locally nilpotent, let $A/S$ be an abelian scheme coming from an $S$-valued point of $S_{K^p}$. Looking at its $p$-divisible group, we get a decomposition
\[A[p^\infty]=(H_1\oplus \cdots\oplus H_r \oplus H')\oplus (H_1\oplus \cdots\oplus H_r \oplus H')^D,\] where for $1\leq i\leq r$, $H_i$ is a $\vp_i$-divisible $O_{D_i}$-module, $H'$ is a $D'$-group in the sense of definition 4.1 of \cite{Sch2} which is the sum of the \'etale $\vp_i$-divisible $O_{D_i}$-modules for $i=r+1,\dots,s$. In particular, after fixing a point in the special fiber of $S_{K^p}$ we can consider the associated (formal) Rapoport-Zink space $\wh{\M}$, which has the decomposition (cf. \cite{RZ} proposition 6.49)
\[\begin{split}
\wh{\M}&\simeq \wh{\M}_{Dr,\vp_1}\times\cdots\times\wh{\M}_{Dr,\vp_r}\times G_{D'}(\Q_p)/O_{D'}^\times \times \Q_p^\times/\Z_p^\times\\
&\simeq (\prod_{i=1}^r \wh{\Omega}^n_{F_{\vp_i}}\times SpfO_{\check{E}}\times D_i^\times/K_{p,\vp_i}^0)\times G_{D'}(\Q_p)/O_{D'}^\times \times \Q_p^\times/\Z_p^\times\\
&\simeq (\prod_{i=1}^r \wh{\Omega}^n_{F_{\vp_i}}\times SpfO_{\check{E}})\times G(\Q_p)/K_p^0,
\end{split}\]
where for $i=1,\dots,r$, $\wh{\M}_{Dr,\vp_i}$ is the formal Drinfeld moduli space associated to the local data, $K_{p,\vp_i}^0\subset D^\times_i$ is the maximal open compact subgroup and $\wh{\Omega}^n_{F_{\vp_i}}$ is the formal Drinfeld upper half space over $SpfO_{F_{\vp_i}}$. Here $\check{E}$ is the completion of the maximal unramified extension of $E$. The associated reductive group $J_b$ has the form (cf. \cite{RZ} 6.44 p. 310, 6.46 and 6.49) \[J_b\simeq Res_{F_{\vp_1}}GL_n\times\cdots\times Res_{F_{\vp_r}}GL_n\times G_{D'}\times \G_m.\]
Now we have the following theorem which says that our Shimura varieties admit global $p$-adic uniformization.
\begin{theorem}[\cite{RZ} Theorem 6.50]
As $K^p$ varies, there is a $G(\A_f^p)$-equivariant isomorphism of formal schemes
\[\coprod_{ker^1(\Q,G)}I(\Q)\setminus(\prod_{i=1}^r \wh{\Omega}^n_{F_{\vp_i}}\times SpfO_{\check{E}})\times G(\A_f)/K \simeq \wh{S}_{K^p}\times SpfO_{\check{E}},\]
where $K=K_p^0K^p$ and $\wh{S}_{K^p}$ is the formal completion of $S_{K^p}$ along its special fiber. The group $I$ is an inner form of $G$ over $\Q$ such that $I(\Q_p)=J_b(\Q_p),\, I(\A_f^p)=G(\A_f^p)$. This defines the action of $I(\Q)$ used in forming the quotient above. The natural descent datum on the right hand side induces on the left hand side the natural descent datum on the first $r$ factors (under the above decomposition for $\wh{\M}$) multiplied with an action of some explicit element $g\in G(\Q_p)$ on $G(\A_f)/K$.
\end{theorem}

Passing to rigid analytic fibers, we get a rigid analytic uniformization of these Shimura varieties $Sh_{K_p^0K_p}$. Moreover, in the rigid analytic setting, we have the uniformization for arbitrary levels at $p$. More precisely, let $m\geq 1$ be an integer, we consider open compact subgroups of the form
\[K_p^m=(1+\Pi_1^mO_{D_1})\times\cdots\times(1+\Pi^m_rO_{D_r})\times (1+p^mO_{D'})\times(1+p^m\Z_p)\subset G(\Q_p),\]where $\Pi_i\in D_i$ is a fixed uniformizer for each $1\leq i\leq r$.
Then we have the following $G(\A_f)$-equivariant isomorphism of rigid analytic spaces
\[I(\Q)\setminus (\prod_{i=1}^r\M_{Dr,\vp_i,m})\times G_{D'}(\Q_p)/(1+p^mO_{D'})\times \Q_p^\times/(1+p^m\Z_p)\times G(\A_f^p)/K^p \simeq Sh_{K_p^mK^p}^{rig}\times \check{E},\]
where $\M_{Dr,\vp_i,m}$ is the level $m$ rigid analytic Drinfeld moduli space for each $1\leq i\leq r$. The above isomorphism is compatible with the two descent datums as in the above theorem on both sides.

\section{The cohomology of Shimura varieties I}
In this section we will assume $r=s=1$ and compute the $\ell$-adic cohomology of these $p$-adically uniformized Shimura varieties. The main ingredients are the Hochschild-Serre spectral sequences (\cite{H1}, \cite{F}), Boyer's description of the $\ell$-adic cohomology of the Lubin-Tate tower (\cite{Bo}), the Fargues-Faltings isomorphism for the towers of Lubin-Tate and Drinfeld (\cite{Fa}, \cite{F1}), and Dat's results for extensions of elliptic representations (\cite{Dat}). We note that in the previous works \cite{H1, H2, H3}, Harris had studied the supercuspidal part of the cohomology. In \cite{H1, H3} he did not prove that the associated local Galois representations are given by the local Langlands correspondences. In \cite{H2} he could prove this key fact for some special case $n<p$. In fact, Harris just constructed the local Langlands correspondences in these cases by the cohomology of these $p$-adically uniformized Shimura varieties. Later, in \cite{HT} Harris-Taylor studied the supercuspidal part of the $\ell$-adic cohomology of some other simple Shimura varieties to construct the local Langlands correspondence for $GL_n$ in the general case.

Fix a prime $l\neq p$. In this paragraph $r$ and $s$ are not necessary 1. Let $\xi$ be an irreducible representation of $G$ over $\ov{\Q}_l$. By standard construction, we have a $\ov{\Q}_l$-local system $\mathcal{L}_\xi$ on each Shimura variety $Sh_K$ for $K\subset G(\A_f)$. We are interested in the virtual $G(\A_f)\times W_E$-representation defined by the alternating sum of $\ell$-adic cohomology
\[H_\xi=\sum_{i\geq 0}(-1)^i\varinjlim_{K}H^i(Sh_K\times\ov{\Q}_p, \mathcal{L}_\xi),\]
where $K$ runs over the set of open compact subgroups of $G(\A_f)$.

Let the notations be as in the previous section with $r=s=1$. For $m\geq 1$, we have the rigid Rapoport-Zink space
\[\M_{K_p^m}=\M_{Dr,\vp,m}\times \Q_p^\times/(1+p^m\Z_p).\]Then the $p$-adic uniformization of $Sh_{K_p^mK^p}$ gives rise to the following spectral sequence (cf. \cite{H1} lemma 6, \cite{F} th\'eor\`eme 4.5.12)
\[\tr{Ext}^i_{J_b(\Q_p)}(H_c^{2(n-1)-j}(\M_{K_p^m}\times \C_p, \ov{\Q}_l(n-1)), \mathcal{A}(I)^{K^p}_\xi)\Rightarrow H^{i+j}(Sh_{K_p^mK^p}\times\ov{\Q}_p, \mathcal{L}_\xi),\]where $\mathcal{A}(I)_\xi$ is the space of automorphic forms on $I$ such that each automorphic representation $\Pi\subset \mathcal{A}(I)_\xi$ has archimedean component $\Pi_\infty=\check{\xi}$, the dual representation of $\xi$; $\mathcal{A}(I)^{K^p}_\xi$ is the $K^p$-invariant subspace, and the Ext is taken in the category of smooth $\ov{\Q}_l$-representations of $J_b(\Q_p)$. For any $\Pi\subset \mathcal{A}(I)_\xi$, write its $p$-component as $\Pi_p$ and the restricted tensor product of its finite components outside $p$ as $\Pi^p$.  Taking direct limits for levels on both sides and passing to the alternating sum, we get the equalities of virtual representations of $G(\A_f)\times W_E$ (cf. \cite{F} corollaire 4.6.3)
\[\begin{split}H_\xi&=\sum_{i,j\geq 0} (-1)^{i+j}\varinjlim_{m}\tr{Ext}^i_{J_b(\Q_p)}(H_c^{j}(\M_{K_p^m}\times \C_p, \ov{\Q}_l(n-1)), \mathcal{A}(I)_\xi)\\
&=\sum_{\begin{subarray}{c}i,j\geq 0\\ \Pi\subset \mathcal{A}(I)_\xi\end{subarray}} (-1)^{i+j}\varinjlim_{m}\tr{Ext}^i_{J_b(\Q_p)}(H_c^{j}(\M_{K_p^m}\times\C_p, \ov{\Q}_l(n-1)),\Pi_p)\otimes\Pi^p.\end{split}\]
This formula can also be viewed as an analogue of Mantovan's formula, cf. \cite{Man} and the last paragraph of this section.
We would like to rewrite the above last formula in a slightly finer form. Recall that under our assumption $r=s=1$, the local reflex field $E=F_\vp$ and the local reductive group has the form $G_{\Q_p}=Res_{E|\Q_p}D^\times\times\G_m$. The cocharacter $\mu$ associated to the Shimura data factors as $(\mu_0,\mu_1)$ with $\mu_0$ (resp. $\mu_1$) the cocharacter of $Res_{E|\Q_p}D^\times$ (resp. $\G_m$). Recall associated to the cocharacter $\mu$ we have the representation $r_{-\mu}$ of the Langlands dual group $^L(G_{E})$ introduced in \cite{Ko1} lemma 2.1.2. We have also representations $r_{-\mu_0}$ and $r_{-\mu_1}$ associated to $\mu_0$ and $\mu_1$ respectively. An irreducible smooth representation $\Pi_p$ of $I(\Q_p)=J_b(\Q_p)$ decomposes as $\Pi_{p,0}\otimes\chi_p$ where $\Pi_{p,0}$ is an irreducible smooth representation of $GL_n(E)$ and $\chi_p$ is a character of $\Q_p^\times$.
\begin{lemma}
We have the equality
\[H_\xi=\sum_{\begin{subarray}{c}i,j\geq 0\\ \Pi\subset \mathcal{A}(I)_\xi\end{subarray}} (-1)^{i+j}\varinjlim_{m}\tr{Ext}^i_{GL_n(E)}(H_c^{j}(\M_{Dr,m}\times \C_p, \ov{\Q}_l(n-1)), \Pi_{p,0})\otimes r_{-\mu_1}\circ\varphi_{\chi_p}|_{W_{E}}\otimes\chi_p\otimes\Pi^{p},\]where $\varphi_{\chi_p}: W_{\Q_p}\lra\, ^L(\G_m)$ is the Langlands parameter associated to $\chi_p$.
\end{lemma}
\begin{proof}
Indeed, by the notation of \cite{Shin}
\[\tr{Mant}_{\mu}(\Pi_p):=\sum_{i,j\geq 0}(-1)^{i+j}\varinjlim_{m}\tr{Ext}^i_{J_b(\Q_p)}(H_c^{j}(\M_{K_p^m}\times \C_p, \ov{\Q}_l(n-1)),\Pi_p),\]we have
\[\tr{Mant}_{\mu}(\Pi_p)=\tr{Mant}_{\mu_0}(\Pi_{p,0})\otimes\tr{Mant}_{\mu_1}(\chi_p),\]with similar definitions of
$\tr{Mant}_{\mu_0}(\Pi_{p,0})$ and $\tr{Mant}_{\mu_1}(\chi_p)$ using the corresponding Rapoport-Zink spaces. This equality comes from the decomposition of our Rapoport-Zink spaces $\M_{K_p^m}$. Now local class field theory tells us that
\[\tr{Mant}_{\mu_1}(\chi_p)=r_{-\mu_1}\circ\varphi_{\chi_p}|_{W_{E}}\otimes\chi_p,\]which can also be viewed as a corollary of the results of \cite{HT} in the case $n=1$, cf. \cite{Shin} proof of theorem 7.5. Now the lemma follows.
\end{proof}

We need to know the $\ell$-adic cohomology of the tower of Drinfeld spaces. To this end, consider the tower of Lubin-Tate spaces $(\M_{LT,K})_{K\subset GL_n(E)}$, where $K$ runs over the set of open compact subgroups of $GL_n(E)$ contained in $GL_n(O_E)$. We have the following isomorphism.
\begin{theorem}[\cite{Fa}, \cite{F1}]
For each $i\geq 0$, there is a $GL_n(E)\times D^\times \times W_E$-equivariant isomorphism of groups
\[\varinjlim_{K}H_c^i(\M_{LT,K}\times\C_p,\ov{\Q}_l)\simeq \varinjlim_{m}H_c^i(\M_{Dr,m}\times\C_p,\ov{\Q}_l).\]
\end{theorem}
In fact, for $0\leq i\leq n-2$ we know the left hand side vanishes, so the right hand side also vanishes in these cases. We denote both sides by $H_c^i$. In \cite{HT} Harris-Taylor computed the supercuspidal part of the cohomology of Lubin-Tate spaces. Boyer in \cite{Bo} has computed the remaining part. It involves elliptic representations of $GL_n(E)$. Recall an elliptic representation of $GL_n(E)$ is an irreducible smooth representation which has the same supercuspidal support as a discrete series representation. For more details we refer to section 2 of \cite{Dat}. We state Boyer's results in the form as in 4.1 of \cite{Dat}. For any irreducible smooth representation $\pi$ of $GL_n(E)$, let $\sigma(\pi)$ be its associated Weil-Deligne representation of $W_{E}$ by the local Langlands correspondence \cite{HT}. If $\pi$ is a discrete series representation, let $\sigma'(\pi)$ be the unique irreducible sub Weil-Deligne representation of $\sigma(\pi)$. The set of irreducible smooth representations of $D^\times$ will be denoted by $Irr(D^\times)$, and an element of it will be usually denoted by $\rho$, with its contragredient denoted by $\rho^\vee$.
\begin{theorem}[\cite{Dat} Th\'eor\`eme 4.1.2]
For $0\leq i\leq n-1$, there is an isomorphism of $GL_n(E)\times D^\times\times W_E$-representations
\[H_c^{n-1+i}\simeq \bigoplus_{\rho\in Irr(D^\times)}\pi_\rho^{\leq i}\otimes\rho\otimes \sigma'(JL(\rho)^\vee)|-|^{\frac{n/n_\rho-n}{2}-i+\frac{1-n}{2}},\]
where $JL(\rho)$ is the discrete series representation of $GL_n(E)$ associated to $\rho$ by the Jacquet-Langlands correspondence, $\pi_\rho|-|^{\frac{1-n_\rho}{2}}\otimes \cdots \otimes \pi_\rho|-|^{\frac{n_\rho-1}{2}}$ is its supercuspidal support, $n_\rho$ is an integer which divides $n$ such that $\pi_\rho$ is the (supercuspidal) representation of $GL_{n/n_\rho}(E)$, $\pi_\rho^{\leq i}$ is the elliptic representation associated to $i$ with the same supercuspidal support as $\pi_\rho$ (cf. \cite{Dat} 2.1.11 and 4.1.1).
\end{theorem}
Here we have corrected the upper subscript of $|-|$ in \cite{Dat} and \cite{Bo} according to theorem VII.1.5 of \cite{HT}. We note that $\sigma(JL(\rho))$ is the Weil-Deligne representation associated to $\rho\in Irr(D^\times)$ by the local Langlands correspondence for the inner form $D^\times$ of $GL_n$, as expected naturally, cf. \cite{ABPS, HS}.

Combining lemma 3.1 with theorem 3.3, we can now prove the following theorem. Since the local reductive group $G_{\Q_p}$ is a product of inner forms of $GL_n$ (together with $\G_m$), the local Langlands correspondence for it is known (\cite{ABPS, HS}). For any smooth irreducible representation $\pi_p$ of $G(\Q_p)$, let $\varphi_{\pi_p}: W_{\Q_p}\lra\,^L(G_{\Q_p})$ be the associated local Langlands parameter.
\begin{theorem}
With the notations as above, we have an identity
\[H_\xi=\sum_{\pi_f}a(\pi_f)\pi_f\otimes(r_{-\mu}\circ\varphi_{\pi_p}|_{W_E})|-|^{(1-n)/2}\]
as virtual $G(\A_f)\times W_E$-representations. Here $\pi_f$ runs through irreducible admissible representations of $G(\A_f)$, the integer $a(\pi_f)$ is as in \cite{Ko3} p. 657.
\end{theorem}
\begin{proof}
Recall that \[\begin{split}&G(\Q_p)=D^\times\times\Q_p^\times,\\& I(\Q_p)=J_b(\Q_p)=GL_n(E)\times \Q_p^\times,\\&G(\A_f^p)=I(\A_f^p).\end{split}\]
For any irreducible representation $\Pi_f$ of $I(\A_f)$, let $\Pi_p$ (resp. $\Pi^p$) be its $p$-component (resp. component outside $p$). As in the paragraph preceding lemma 3.1, we denote $\Pi_p=\Pi_{p,0}\otimes\chi_p$, with $\Pi_{p,0}$ an irreducible smooth representation of $GL_n(E)$, $\chi_p$ a character of $\Q_p^\times$. Then for any irreducible representation $\rho$ of $D^\times$, the tensor product \[\rho\otimes\Pi^{p,0}\] gives us an irreducible representation $\pi_f$ of $G(\A_f)$, and vice versa.  By lemma 3.1 we have
\[\begin{split}H_\xi&=\sum_{\begin{subarray}{c}i,j\geq 0\\ \Pi\subset \mathcal{A}(I)_\xi\end{subarray}} (-1)^{i+j}\varinjlim_{m}\tr{Ext}^i_{GL_n(E)}(H_c^{j}(\M_{Dr,m}\times \C_p, \ov{\Q}_l(n-1)), \Pi_{p,0})\otimes r_{-\mu_1}\circ\varphi_{\chi_p}|_{W_{E}}\otimes\chi_p\otimes\Pi^p\\
&=\sum_{\Pi\subset\mathcal{A}(I)_\xi}\tr{Mant}_{\mu_0}(\Pi_{p,0})\otimes
r_{-\mu_1}\circ\varphi_{\chi_p}|_{W_{E}}\otimes\chi_p\otimes\Pi^p.\end{split}\]
As the notation in the proof of lemma 3.1, we have to compute
\[\tr{Mant}_{\mu_0}(\Pi_{p,0})=\sum_{i,j\geq0}(-1)^{i+j}\varinjlim_{m}\tr{Ext}^i_{GL_n(E)}(H_c^{j}(\M_{Dr,m}\times \C_p, \ov{\Q}_l(n-1)), \Pi_{p,0}).\]First, we rewrite it as
\[\tr{Mant}_{\mu_0}(\Pi_{p,0})=\sum_{\begin{subarray}{c}i\geq 0\\k\geq 0\\ \end{subarray}} (-1)^{i+n-1+k}\varinjlim_{m}\tr{Ext}^i_{GL_n(E)}(H_c^{n-1+k}(n-1)^{K_{p,0}^m}, \Pi_{p,0}).\]Apply the formula in theorem 3.3 for $H_c^{n-1+k}$ (here over $\ov{\Q}_l$ we can ignore the Tate twist), then take the terms other than $\pi_\rho^{\leq k}$ out of the Ext, and then take the direct limit on $m$ we get that $\tr{Mant}_{\mu_0}(\Pi_{p,0})$ equals
\[\sum_{\rho\in Irr(D^\times)}\sum_{\begin{subarray}{c}i\geq0\\n_\rho-1\geq k\geq0\\ \end{subarray}} (-1)^{i+n-1+k}\tr{Ext}^i_{GL_n(E)}(\pi_\rho^{\leq k}, \Pi_{p,0})\otimes\rho\otimes \sigma'(JL(\rho)^\vee)|-|^{\frac{n/n_\rho-n}{2}-k+\frac{1-n}{2}}.\]
Now by proposition 2.1.17 of \cite{Dat}, for $0\leq k\leq n_\rho-1$, $\tr{Ext}^i_{GL_n(E)}(\pi_\rho^{\leq k}, \Pi_{p,0})\neq 0$ if and only if $\Pi_{p,0}$ is elliptic and has the supercuspidal support $\pi_\rho$ and $i$ satisfies an equality which depends on $k$ and $\Pi_{p,0}$. In this case, $\tr{Ext}^i_{GL_n(E)}(\pi_\rho^{\leq k}, \Pi_{p,0})=\ov{\Q}_l$, and since $\Pi_{p,0}$ is preunitary for any fixed isomorphism $\ov{\Q}_l\simeq\C$, it has to be $\pi_{\rho}^{\leq 0}:=JL(\rho)$ or the local Speh representation $\pi_\rho^{\leq n_\rho-1}$, see loc. cit. 5.2 p. 140. On the other hand, we know by construction \[\sum_{k=0}^{n_\rho-1}\sigma'(JL(\rho)^\vee)|-|^{\frac{n/n_\rho-n}{2}-k}=\sigma(JL(\rho)^\vee).\]Therefore, we can continue the above formula as follows:
\[\begin{split}H_\xi=&\sum_{\rho\in Irr(D^\times)}\sum_{\begin{subarray}{c}\Pi\subset \mathcal{A}(I)_\xi\\\Pi_{p,0}=JL(\rho)\end{subarray}}(-1)^{n-1}\rho\otimes \sigma(JL(\rho)^\vee)|-|^{\frac{1-n}{2}}\otimes r_{-\mu_1}\circ\varphi_{\chi_p}|_{W_{E}}\otimes \chi_p\otimes\Pi^p\\
&+\sum_{\rho\in Irr(D^\times)}\sum_{\begin{subarray}{c}\Pi\subset \mathcal{A}(I)_\xi\\\Pi_{p,0}=\pi_\rho^{\leq n_\rho-1}\end{subarray}}(-1)^{n-1+n_\rho-1}\rho\otimes \sigma(JL(\rho)^\vee)|-|^{\frac{1-n}{2}}\otimes r_{-\mu_1}\circ\varphi_{\chi_p}|_{W_{E}}\otimes\chi_p\otimes\Pi^p\\
=&\sum_{\rho\in Irr(D^\times)}\sum_{\begin{subarray}{c} \Pi\subset \mathcal{A}(I)_\xi\\\Pi_{p,0}=JL(\rho)\,\tr{or}\,\pi_\rho^{\leq n_\rho-1}\end{subarray}}(-1)^{d_\rho}\rho\otimes \sigma(JL(\rho)^\vee)|-|^{\frac{1-n}{2}}\otimes r_{-\mu_1}\circ\varphi_{\chi_p}|_{W_{E}}\otimes\chi_p\otimes\Pi^p,\end{split}\]
where $d_\rho\in\{0,1\}$ depends on $\rho$. For $\rho,\Pi$ occurring in the above formula, let $\pi_f=\rho\otimes\chi_p\otimes\Pi^p$. Then $\pi_p=\rho\otimes\chi_p$, and
\[r_{-\mu}\circ\varphi_{\pi_p}|_{W_E}=r_{-\mu_0}\circ\varphi_{\rho}|_{W_E}\otimes r_{-\mu_1}\circ\varphi_{\chi_p}|_{W_E}.\]

By comparing with Matsushima's formula, the above formula for $H_\xi$ equals
\[\sum_{\pi_f}a(\pi_f)\pi_f\otimes(r_{-\mu}\circ\varphi_{\pi_p}|_{W_E})|-|^{(1-n)/2},\]with the integer $a(\pi_f)$ as in \cite{Ko3} p. 657.
Note that we  have \[a(\pi_f)=\sum_{\rho\in Irr(D^\times)}\sum_{\begin{subarray}{c}\Pi\subset \mathcal{A}(I)_\xi\\ \Pi_{p,0}=JL(\rho)\,\tr{or}\,\pi_\rho^{\leq n_\rho-1}\\ \pi_f=\rho\otimes \chi_p\otimes \Pi^p\end{subarray}}(-1)^{d_\rho}.\]
Here we have used the explicit description of \[\begin{split}r_{-\mu_0}:\, ^L(D^\times_{E})=GL_n(\ov{\Q}_l)\times (\prod_{\begin{subarray}{c}\tau':E\hookrightarrow\ov{\Q}_p\\\tau'\neq\tau\end{subarray}} GL_n(\ov{\Q}_l))\times W_{E}&\longrightarrow GL_n(\ov{\Q}_l)\\ (g,g',\sigma)&\longmapsto (g^{-1})^t.\end{split}\] Hence we have
\[r_{-\mu_0}\circ\varphi_{\rho}|_{W_E}=\sigma(JL(\rho))^\vee=\sigma(JL(\rho)^\vee).\]

\end{proof}
In \cite{Man} Mantovan introduced Igusa varieties for all PEL type Shimura varieties with the related reductive groups unramified at $p$, and established a formula which describes the cohomology of each Newton strata by the cohomology of the associated Rapoport-Zink spaces and Igusa varieties. Here in our setting, the space $\mathcal{A}(I)_\xi$ is a substitute of the cohomology of the Igusa varieties, as one can define in a similar way to \cite{Man}. However, the space $\mathcal{A}(I)_\xi$ should be slightly larger, since one should have a similar formula as theorem 6.7 of \cite{Shin} which says that, the cohomology of the Igusa varieties should be understood from automorphic representations of $G$ by the Jacquet-Langlands correspondence, therefore locally at $p$ the representations are only restricted in the class of all discrete series of $GL_n$.

\section{Deformation spaces of special $O_D$-modules and test functions}
We want to prove the above theorem for all $p$-adically uniformized Shimura varieties introduced in section 2, i.e. $r$ can be arbitrary positive integer. To this end, we will follow an another approach: the Langlands-Kottwitz approach for these Shimura varieties at hand. We will apply Scholze's method to define some test functions by means of deformation spaces of $p$-divisible groups, cf. \cite{Sch1, Sch2, Sch3}. Our local setting is in the EL case as in \cite{RZ}, which is not included in \cite{Sch3} since there one restricts to the general linear groups.

We change the notations in this section. We will not make the full generality as in \cite{RZ} 1.38 and definition 3.18. Here we will restrict ourself to the simple EL case as in loc. cit.. The general case can be studied in the same way, or by working with each simple factors as presented here. In fact, for the purpose of this paper, only a more restricted case will be used later. Let $D$ be a central division algebra of dimension $n^2$ over a finite extension $F$ of $\Q_p$, with invariant $\frac{s}{n}$. Let $V$ be a finite left $D$-module. We fix a maximal order $O_D\subset D$ and a $O_D$-stable lattice $\Lambda\subset V$. These data give use the semisimple $\Q_p$-algebra $C=End_D(V)$ with the maximal order $O_C=End_{O_D}(\Lambda)$. Let $G/\Z_p$ be the algebraic group whose group of $R$-valued points is given by \[G(R)=(R\otimes_{\Z_p}O_C)^\times\]for any $\Z_p$-algebra $R$. Let $\{\mu\}$ be the conjugacy class of the cocharacter \[\mu: \G_m\lra G_{\ov{\Q}_p}.\] The field of definition of $\{\mu\}$ is denoted by $E$. Fix a representative $\mu$ of $\{\mu\}$
over $\ov{\Q}_p$. We assume that only weights 0 and 1 occur in the associated decomposition of $V$ over $\ov{\Q}_p$, i.e. $V_{\ov{\Q}_p}=V_0\oplus V_1$. The isomorphism class of the subspace $V_0$ (and $V_1$) is defined over $E$. We make the following definition of special $O_D$-modules as $p$-divisible groups with suitable actions of $O_D$ (cf. \cite{Sch3} definition 3.3 or \cite{RZ} 3.23).
\begin{definition}
Let $S$ be a scheme over $O_E$ on which $p$ is locally nilpotent. A special $O_D$-module is given by a pair $\ul{H}=(H,\iota)$ consisting of
\begin{itemize}
\item a $p$-divisible group $H$ over $S$,
\item a homomorphism $\iota: O_D\ra End(H)$ such that
\begin{enumerate}
\item locally on $S$ there is an isomorphism of $O_D\otimes O_S$-modules between the Lie algebra of the universal vector extension of $H$ and $O_D\otimes_{\Z_p}O_S$, and
\item the determinant condition holds true, i.e. we have an identity of polynomial functions in $a\in O_D$
\[det_{O_S}(a|LieH)=det_E(a|V_0).\]
\end{enumerate}
\end{itemize}
\end{definition}

With the above definition, many results of sections 3 and 4 in \cite{Sch3} still hold true in our context. We just review and summarize what we need, for details we refer to loc. cit.. Let $\ul{H}=(H,\iota)$ be a special $O_D$ module over a perfect field $\kappa$ of characteristic $p$, which we give a structure of an $O_E$-algebra via a fixed morphism $O_E\ra \kappa$.
\begin{itemize}
\item The deformation functor $Def_{\ul{H}}$ of $\ul{H}$ as a special $O_D$-module is pro-representable by a complete noetherian local $O_E$-algebra $R_{\ul{H}}$ with residue field $\kappa$.
\item Let $k^0$ be the complete discrete valuation ring with residue field $\kappa$ that is unramified over $O_E$, and let $k$ be its fraction field. Then $R_{\ul{H}}$ has a structure as a $k^0$-algebra. Let $X_{\ul{H}}=(SpfR_{\ul{H}})^{rig}$ be the rigid generic fiber of $SpfR_{\ul{H}}$, as a rigid analytic space over $k$. Then for any open compact subgroup $K\subset G(\Z_p)$, we have a finite \'etale covering $X_{\ul{H},K}$ of $X_{\ul{H}}$ parameterizing level $K$ structures.
\item If $X_{\ul{H}}\neq \emptyset$, then $\kappa_G(b)=\mu^{\sharp}$, where $b\in B(G_{\Q_p})$ is the $\sigma$-conjugacy class defined from the Frobenius morphism on the covariant Dieudonn\'e module of $\ul{H}$, $\kappa_G: B(G_{\Q_p})\lra X^\ast(Z(\wh{G})^\Gamma)$ is the Kottwitz map defined in \cite{Ko5}, $\mu^\sharp\in X^\ast(Z(\wh{G})^\Gamma)$ is the element defined from the conjugacy class of cocharacters $\mu$ as in loc. cit.. Here $B(G_{\Q_p})$ is the set of $\sigma$-conjugacy class in $G(W(\ov{\kappa})_\Q)$, $\wh{G}$ is the dual group, $Z(\wh{G})$ its center, and $\Gamma=Gal(\ov{\Q}_p/\Q_p)$.
\item For any perfect field $\kappa$ of characteristic $p$ which is an $O_E$-algebra, there is an association by using Dieudonn\'e module theory $\ul{H}\mapsto \delta\in G(W(\kappa)_\Q)$ which defines an injection from the set of isomorphism classes of special $O_D$-modules over $\kappa$ such that $X_{\ul{H}}\neq\emptyset$ into the set of $G(W(\kappa))$-$\sigma$-conjugacy classes in $G(W(\kappa)_\Q)$ with the properties $pO_D\subset p\delta O_D\subset O_D$ and $\kappa_G(p\delta)=\mu^{\sharp}$.
\item We say $\ul{H}$ has controlled cohomology if $X_{\ul{H},K}$ has controlled cohomology for all normal pro-p open subgroups $K\subset G(\Z_p)$ and all primes $l\neq p$ in the sense of definition 2.2 of \cite{Sch3}. Assume $\ul{H}$ has controlled cohomology. Then for any normal pro-p open subgroup $K\subset G(\Z_p)$, there is an integer $m\geq 1$ such that for all automorphisms $j$ of $\ul{H}$ that act trivially on $\ul{H}[p^m]$ the induced action on $H^i(X_{\ul{H},K}\times\hat{\ov{k}},\Q_l)$ is trivial for all $i$.
\end{itemize}

Let $I_E\subset W_E$ be the inertia subgroup of the Weil group, and fix a geometric Frobenius element $Frob\in W_E$. Fix some integer $j\geq 1$. Let $\tau\in Frob^jI_E\subset W_E$ and $h\in C_c^\infty(G(\Z_p))$ be a function with values in $\Q$. Set $t=j[\kappa_E: \F_p]$ where $\kappa_E$ is the residue field of $E$ (the standard notation for $j[\kappa_E: \F_p]$ should be $r$, but we have used $r$ as the copies of Drinfeld spaces in the $p$-adic uniformization). We regard $\F_{p^t}$ as the degree-$j$-extension of $\kappa_E$. As before, let $k$ be the unramified extension of $E$ with residue field $\F_{p^t}$. Fix the Haar measures on $G(\Q_p)$, resp. $G(\Q_{p^t})$, that give $G(\Z_p)$ resp. $G(\Z_{p^t})$ volume 1.
\begin{definition}Let $\delta\in G(\Q_{p^t})$. If $\delta$ is associated to some special $O_D$-module $\ul{H}$ over $\F_{p^t}$ under the above association, and if $\ul{H}$ has controlled cohomology, define
\[\phi_{\tau,h}(\delta)=tr(\tau\times h| H^\ast(X_{\ul{H},K}\times\hat{\ov{k}},\Q_l))\]for any normal compact pro$-p$ open subgroup $K\subset G(\Z_p)$ such that $h$ is $K$-biinvariant. Otherwise, define $\phi_{\tau,h}(\delta)=0$. Here and in the following \[H^\ast(X_{\ul{H},K}\times\hat{\ov{k}},\Q_l)=\sum_{i\geq0}(-1)^iH^i(X_{\ul{H},K}\times\hat{\ov{k}},\Q_l).\]
\end{definition}
\begin{proposition}\begin{enumerate}
\item The function $\phi_{\tau,h}: G(\Q_{p^t})\ra \Q_l$ is well defined and takes values in $\Q$ independent of $l$. Its support is contained in the compact set of all $\delta\in G(\Q_{p^t})$ satisfying $pO_D\subset p\delta O_D\subset O_D$ and $\kappa_G(p\delta)=\mu^{\sharp}$.
\item The function $\phi_{\tau,h}$ is locally constant, so that it defines an element $\phi_{\tau,h}\in C_c^\infty(G(\Q_{p^t}))$.
\end{enumerate}
\end{proposition}
\begin{proof}
Identical to the proofs of propositions 4.2 and 4.3 of \cite{Sch3}.
\end{proof}

We come back to the global situation as in section 2. Recall we have local central division algebras $D_i=B_{\vp_i}$ over $F_{\vp_i}$ with invariant $\frac{1}{n}$ for $1\leq i\leq r$ and the semisimple $\Q_p$-algebra $D'$ which corresponds to the factors of $B$ at the primes $\vp_{r+1},\dots,\vp_s$. For each $1\leq i\leq r$, the cocharacter $\mu_i$ has the form that \[\begin{split} \G_m&\lra G_{i\ov{\Q}_p}\\
z&\mapsto (\left(\begin{array}{cc}
 1 & \\
  &z1_{n-1}\\
  \end{array}\right),1_n,\cdots,1_n),\end{split}\]where $G_i=Res_{F_{\vp_i}|\Q_p}D_i^\times$ and for any integer $d\geq1$, $1_d$ is the identity $d\times d$ matrix. Recall now $E$ is the compositum of the fields $F_{\vp_i}$ for $1\leq i\leq r$. We consider the following $p$-divisible groups.
\begin{definition}
Let $S$ be a scheme over $O_E$ on which $p$ is locally nilpotent. A $(D_1,\dots,D_r,D')$-group over $S$ is a $p$-divisible group $\wt{\ul{H}}=((H_1,\iota_1),\dots,(H_r,\iota_r),(H',\iota'))$ where
\begin{itemize}
\item for each $1\leq i\leq r$, $(H_i,\iota_i)$ is a special $O_{D_i}$-module over $S$,
\item $(H',\iota')$ is a $D'$ group over $S$ in the sense of definition 4.1 of \cite{Sch2}, i.e. an \'etale $p$-divisible group $H'$ over $S$ together with an action $\iota': O_{D'}^{op}\ra End(H')$ such that $H'[p]$ is free of rank 1 over $O_{D'}^{op}/p$.
\end{itemize}
\end{definition}
Recall that there are two ways to parameterize $D'$-groups $H'$ over $\F_{p^t}$, the Dieudonn\'e parametrization and the Galois parametrization. Let $\sigma_0$ be the absolute Frobenius of $\Z_{p^t}$. The set of isomorphism classes of such $D'$-groups is in bijection with the set of $(O_{D'}\otimes_{\Z_p}\Z_{p^t})^\times$-$\sigma_0$ conjugacy classes in $(O_{D'}\otimes_{\Z_p}\Z_{p^t})^\times$, which is in turn bijection with the set of $O_{D'}^\times$-conjugacy classes in $O_{D'}^\times$ by a map $\delta'\mapsto N\delta'$, cf. proposition 4.2 of \cite{Sch2}. Let $h'\in C_c^\infty(O_{D'}^\times)$ be a function which takes values in $\Q$ and invariant under conjugation. We define a function $\phi_{h'}\in C_c^\infty((O_{D'}\otimes_{\Z_p}\Z_{p^t})^\times)$ by setting $\phi_{h'}(\delta')=h'(N\delta')$. Then by proposition 4.3 of loc. cit. the functions $\phi_{h'}$ and $h'$ have matching (twisted) orbital integrals.

Let $\wt{\ul{H}}$ be a $(D_1,\dots,D_r, D')$-group over $\F_{p^t}$, with $t=j[\kappa_E: \F_p]$. Let $G_B$ be the group scheme over $\Z_p$ given by
\[G_B(R)=(O_B\otimes_{\Z_p}R)^\times\] for any $\Z_p$-algebra $R$. Then $G_B\times\G_m$ is a integral model of the $G_{\Q_p}$ introduced in section 2. As in the single factor case, the group $\wt{\ul{H}}$ gives rise to an element
\[\delta=(\delta_1,\dots,\delta_r,\delta')\in G_B(\Q_{p^t})=(D_1\otimes_{\Q_p}\Q_{p^t})^\times\times\cdots\times (D_r\otimes_{\Q_p}\Q_{p^t})^\times\times G_{D'}(\Q_{p^t})\]which is well defined up to $G_B(\Z_{p^t})$-$\sigma_0$-conjugacy. The deformation functor of $\wt{\ul{H}}$ as $(D_1,\dots,D_r,D')$-group is prorepresentable by a $k^0$-algebra $R_{\wt{\ul{H}}}$, where as before $k^0$ is the unramified extension of $O_E$ with residue field $\F_{p^t}$ and fraction field $k$. We assume \[X_{\wt{\ul{H}}}=(SpfR_{\wt{\ul{H}}})^{rig}\]
nonempty, which imposes the condition $\kappa_{G_B}(p\delta)=\mu^\sharp$ on $\delta$.
Moreover, for an open compact subgroup of the form $K=(\prod_{i=1}^rK_i)\times K'\subset G_B(\Z_p)$, we have a product decomposition of the level $K$ covering space
\[X_{\wt{\ul{H}},K}=(\prod_{i=1}^rX_{\ul{H}_i,K_i})\times G_{D'}/K'\]of $X_{\wt{\ul{H}}}$.
Now for $\tau\in Frob^jI_E\subset W_E, \, h_1\in C_c^\infty(O_{D_1}^\times),\dots,h_r\in C_c^\infty(O_{D_r}^\times),h'\in C_c^\infty(O_{D'}^\times)$, with all these functions take values in $\Q$ and $h'$ invariant under conjugation, we have defined the functions $\phi_{\tau, h_1}\in C_c^\infty(D_1^\times(\Q_{p^t})), \dots, \phi_{\tau,h_r}\in  C_c^\infty(D_r^\times(\Q_{p^t}))$ and $\phi_{h'}\in C_c^\infty((O_{D'}\otimes_{\Z_p}\Z_{p^t})^\times)$. We can also define a function $\phi_{\tau,h_1,\dots,h_r,h'}\in C_c^\infty(G_B(\Q_{p^t}))$ by
\[\phi_{\tau,h_1,\dots,h_r,h'}(\delta)=tr(\tau\times h_1\times\cdots\times h_r\times h'|H^\ast(X_{\wt{\ul{H}},K}\times\hat{\ov{k}},\Q_l))\]for any normal compact pro-$p$ open subgroup $K\subset G_B(\Z_p)$ such that $(h_1,\dots,h_r,h')$ is $K$-biinvariant, if $\delta$ is associated to a controlled $(D_1,\dots,D_r,D')$-group $\wt{\ul{H}}$, otherwise define $\phi_{\tau,h_1,\dots,h_r,h'}(\delta)=0$. Then it is easy to see that for $\delta=(\delta_1,\cdots,\delta_r,\delta')$ we have
\[\phi_{\tau,h_1,\dots,h_r,h'}(\delta)=\phi_{\tau,h_1}(\delta_1)\cdots\phi_{\tau,h_r}(\delta_r)\phi_{h'}(\delta').\]

Our global PEL situation actually puts us in the quasi-EL case in the following sense, which is the analogue of the quasi-EL case of \cite{Sch3} section 3 in our context. Let $S$ be a locally noetherian scheme on which $p$ is locally nilpotent. Recall for an abelian scheme $A$ which comes from an $S$-valued point of the moduli space $S_{K^p}$ defined in section 2, the associated $p$-divisible group has the decomposition
\[A[p^\infty]=\wt{\ul{H}}\oplus\wt{\ul{H}}^D,\]
with $\wt{\ul{H}}$ a $(D_1,\dots,D_r,D')$-group, $\wt{\ul{H}}^D$ its dual. Then one can also formulate a definition of $p$-divisible groups over $S$ with the same additional (PEL) structures as those in the form $A[p^\infty]$ for $A$ coming from an $S$-point of $S_{K^p}$, cf. \cite{Sch3} definition 3.3. We may call these additional structures as quasi-EL additional structures. Such a $p$-divisible group with quasi-EL additional structures is given by a quadruple $\ul{H}=(H,\iota,\lambda,\mathbb{L})$ over $S$, with $\lambda: H\st{\sim}{\lra}H^D\otimes\mathbb{L}$ a (twisted) principal polarization and $\mathbb{L}$ is a one-dimensional smooth $\Z_p$-local system. It is equivalent to giving a $(D_1,\dots,D_r,D')$-group $\wt{\ul{H}}_0$ and an one-dimensional smooth $\Z_p$-local system $\mathbb{L}$ over $S$. The relation between the two is given by the equality $H=H_0\oplus H_0^D\otimes\mathbb{L}$. In particular, for such a $p$-divisible group $\ul{H}$ over $\F_{p^t}$, we get an element
\[\delta=(\delta_0,p^{-1}d)\in G(\Q_{p^t})=G_B(\Q_{p^t})\times\Q_{p^t}^\times\]where $\mathbb{L}$ corresponds to $d$ as in section 3 of \cite{Sch3}. Then for any open compact subgroup of the form
\[K=K_0\times (1+p^m\Z_p)\subset G(\Z_p)=G_B(\Z_p)\times\Z_p^\times\] with $m\geq1$, we have a product decomposition \[X_{\ul{H},K}=X_{\wt{\ul{H}}_0,K_0}\times X_{\mathbb{L},m},\]where $X_{\mathbb{L},m}$ parameterizes isomorphisms between $\mathbb{L}\otimes\mu_{p^m}$ and $\Z/p^m\Z$. Similarly for $\tau\in Frob^jI_E\subset W_E, (h_0, h_{\G_m})\in C_c^\infty(G_B(\Z_p))\times C_c^\infty(\Z_p^\times)$ with values in $\Q$, we can define a test function $\phi_{\tau,h_0,h_{\G_m}}\in C_c^\infty(G(\Q_{p^t}))$ such that for all $\delta=(\delta_0,\delta_{\G_m})\in G(\Q_{p^t})=G_B(\Q_{p^t})\times\Q_{p^t}^\times$, we have
\[\phi_{\tau,h_0,h_{\G_m}}(\delta)=\phi_{\tau,h_0}(\delta_0)\phi_{\tau,h_{\G_m}}(\delta_{\G_m}),\]
where $\phi_{\tau,h_{\G_m}}$ is the function with support on $p^{-1}\Z_{p^t}^\times$ defined by
\[\phi_{\tau,h_{\G_m}}(\delta_{\G_m})=h(Art_{\Q_p}(\tau)N\delta_{\G_m}),\]where $Art_{\Q_p}: W_{\Q_p}\ra \Q_p^\times$ is the local reciprocity map sending a geometric Frobenius element to a uniformizer. As in the remark 4.11 of \cite{Sch3}, for all characters $\chi: \Q_p^\times\ra\C^\times$, we have the identity
\[tr(\phi_{\tau,h_{\G_m}}|\chi\circ N_{\Q_{p^t}|\Q_p})=tr(\tau^{-1}|\chi\circ Art_{\Q_p})tr(h|\chi),\]where $N_{\Q_{p^t}|\Q_p}:\Q_{p^t}^\times\ra \Q_p^\times$ is the norm map.

\section{Counting points modulo $p$}
We return to the situation of section 2. We are going to adapt Kottwitz's method to count points in $S_{K^p}(\F_{p^t})$ and compute the cohomology of Shimura varieties. Before going into the details, we say more about the moduli spaces $S_{K^p}$.

We know that \[S_{K^p}\times_{O_E}E\simeq \coprod_{ker^1(\Q,G)}Sh_{K_p^0K^p},\]where $K_p^0\subset G(\Q_p)$ is the maximal open compact subgroup. Now for any open compact subgroup $K_p\subset G(\Z_p)$, we have a finite \'etale cover $\pi_{K_pK^p}: S_{K_pK^p}\ra S_{K^p}\times_{O_E}E$ which parameterizes $K_p$-orbits of isomorphisms between $O_B\otimes\Z_p$ and the $p$-adic Tate module $T_pA$ of $A$, compatible with the additional structures. In such a way we get a tower of varieties equipped with an action of $G(\Z_p)\times G(\A_f^p)$ by Hecke correspondences. Moreover, we have isomorphisms
\[S_{K_pK^p}\simeq \coprod _{ker^1(\Q,G)}Sh_{K_pK^p}\]compatible with the Hecke correspondences and the maps to $S_{K^p}\times_{O_E} E$, cf. \cite{Sch3} proposition 5.3.

Fix a prime $l\neq p$. Let $\xi$ be a finite dimensional algebraic representation of $G$ over $\ov{\Q}_l$. Then as usual we get local systems $\mathcal{L}_{\xi}$ on $S_{K_pK^p}$ to which the action of the Hecke correspondences extend. By abuse of notation we will not put level subscripts on them. We can define the $\ell$-adic cohomology of the tower $S_{K_pK^p}$ with coefficients in the local systems $\mathcal{L}_\xi$ by
\[\wt{H}_\xi=\sum_{i\geq 0}(-1)^i\varinjlim_{K_p,K^p}H^i(S_{K_pK^p}\times \ov{\Q}_p,\mathcal{L}_\xi).\]
This is a virtual representation of $G(\Z_p)\times G(\A_f^p)\times W_E$. We have the equality
\[\wt{H}_\xi=\bigoplus_{ker^1(\Q,G)}H_\xi\]
which is compatible with the actions of $W_E$ and $G(\Z_p)\times G(\A_f^p)$, where $H_\xi$ was defined in section 3.

Now let $x\in S_{K^p}(\F_{p^t})$, write $\ul{H}$ as the associated $p$-divisible group with quasi-EL additional structures over $\F_{p^t}$ defined in section 4. Let $k$ be the complete unramified extension of $E$ with residue field $\F_{p^t}$. By Serre-Tate theorem, we know that there is an isomorphism of complete local rings $\wh{O}_{S_{K^p},x}\simeq R_{\ul{H}}$, where $R_{\ul{H}}$ is the deformation ring of $\ul{H}$. In particular, we have the natural embedding $X_{\ul{H}}\hookrightarrow S_{K_p^0K^p}^{rig}\times k$. Moreover for any open compact subgroup $K_p\subset G(\Z_p)$ we have the pullback diagram
\[\xymatrix{
X_{\ul{H},K_p}\ar@{^{(}->}[d]\ar[r]&X_{\ul{H}}\ar@{^{(}->}[d]\\
S_{K_pK^p}^{rig}\times k\ar[r]&S_{K_p^0K^p}^{rig}\times k.}\]Here for any $K_p\subset G(\Z_p)$, $S^{rig}_{K_pK^p}$ is the rigid analytic space over $E$ associated to $S_{K_pK^p}$. In particular, $\ul{H}$ has controlled cohomology and for all $i\in\Z$ we have a $Gal(\ov{k}/k)$-equivariant isomorphism
\[(R^i\psi\pi_{K_pK^p\ast}\Q_l)_{\ov{x}}\simeq H^i(X_{\ul{H},K_p}\times \hat{\ov{k}}, \Q_l),\]
where $\pi_{K_pK^p}: S_{K_pK^p}\ra S_{K_p^0K^p}$ is the natural projection, and $\psi$ is the nearby cycle functor for the scheme $S_{K^p}$.

The $p$-adic uniformization gives us the description of the set of $\ov{\F}_p$-points of $S_{K^p}$:
\[S_{K^p}(\ov{\F}_p)=\coprod_{ker^1(\Q,G)}I(\Q)\setminus X_p \times X^p,\]
where $X_p=\M(\ov{\F}_p)$ with $\M$ the pro-formal scheme over $O_E$ associated to the formal Rapoport-Zink space $\wh{\M}$ as in definition 3.51 of \cite{RZ}, $X^p=G(\A_f^p)/K^p$, and the Frobenius action on the left hand side induces the natural Frobenius action on $X_p$. The set $X_p$ can be described as a set of some suitable Dieudonn\'e lattices, and after fixing a choice of a lattice we can identify $X_p$ with a subset of $G(L)/G(O_L)$ where $L=W(\ov{\F}_p)_\Q$. As $K^p$ varies, this bijection is compatible with the action of $G(\A_f^p)$. Note for our group $G$, if $K^p$ sufficiently small we have $Z(\Q)\cap K=1$, where $Z\subset G$ is the center and $K=K_p^0K^p$. From now on, we take a sufficiently small $K^p$ such that $Z(\Q)\cap K=1$, and such that 1.3.7 and 1.3.8 of \cite{Ko2} hold (the proof of lemma 5.5 in \cite{Mi} works for the group $I$, as one sees easily). Passing to the finite level, by the method of \cite{Ko2} we have
\[S_{K^p}(\F_{p^t})=\coprod_{ker^1(\Q,G)}\coprod_{\varepsilon}I_\varepsilon(\Q)\setminus A_\varepsilon,\]
where $\varepsilon$ runs through the set of conjugacy classes in $I(\Q)$, $I_\varepsilon$ is the centralizer of $\varepsilon$ in $I$, and $A_\varepsilon$ is defined as follows:
\[\begin{split}A_\varepsilon&=\tr{Fix}(Fr^j\varepsilon^{-1}|X_p)\times \tr{Fix}(\varepsilon|X^p)\\&=:X_p(\varepsilon)\times X^p(\varepsilon).\end{split}\]
Here $Fr$ is the Frobenius automorphism on $X_p$ over $E$. The set $A_\varepsilon$ could be empty. However, when it is non empty,
the set $\coprod_{ker^1(\Q,G)}I_\varepsilon(\Q)\setminus A_\varepsilon$ has a moduli explanation as follows. Let $x_0=(A_{x_0},\iota_{x_0},\lambda_{x_0},\ov{\eta}_{x_0}^p)$ $\in S_K(\F_{p^t})$ be any point with associated $\varepsilon$. Then the set of points in $S_K^p(\F_{p^t})$ for which the associated ableian varieties are isogenious to $(A_{x_0},\iota_{x_0},\lambda_{x_0})$ over $\F_{p^t}$ is in bijection with $\coprod_{ker^1(\Q,G)}I_\varepsilon(\Q)\setminus A_\varepsilon$. As in the case of $X_p$, after fixing some choice of some lattice, the set $X_p(\varepsilon)$ can be identified with a subset of $G(\Q_{p^t})/G(\Z_{p^t})$.

For $\varepsilon\in I(\Q)$ such that $A_\varepsilon$ is non empty, we can associate to it a conjugacy class $\gamma\in G(\A_f^p)$ that is stably conjugate to $\varepsilon$ by the embedding $I(\Q)\subset I(\A_f^p)=G(\A_f^p)$, and a $\sigma$-conjugacy class $\delta\in G(\Q_{p^t})$ such that $N\delta$ is stably conjugate to $\varepsilon$, cf. \cite{Ra} theorem 4.11 and \cite{Ra2} p. 689. The characterization of the $\sigma$-conjugacy class $\delta$ is as follows. Assume that $A_\varepsilon$ is non empty, then there is a $\delta\in G(\Q_{p^t})$ such that there exists
$c\in G(L)$ with $\delta= cb\sigma(c)^{-1}$ and $N\delta = c\varepsilon c^{-1}$; here $N\delta$ is as in definition \ref{D:Kottwitz} and $b\times \sigma$ is the $\sigma$-linear map $V^{-1}$ on the covariant rational Dieudonn\'e module. The $\sigma$-conjugacy class of such a $\delta\in G(\Q_{p^t})$  is
uniquely determined by the $I(\Q)$-conjugacy class of $\varepsilon$. If $A_\varepsilon$ is non empty, we can also get $(\gamma,\delta)$ by geometric means. A point $x\in S_{K}(\F_{p^t})$ with associated $\varepsilon$ gives rise to a $c$-polarized virtual abelian variety $(A_x,\iota_x,\lambda_x)$ over $\F_{p^t}$, cf. \cite{Sch3} definition 6.1 or section 10 of \cite{Ko2}. By the existence of a level structure of type $K^p$, we know that for all $l\neq p$, the rational $l$-adic Tate module $V_l(A_x)$ is isomorphic to $V\otimes\Q_l$; fixing an isomorphism, the Frobenius morphism $\pi_{A_x}$ gives rise to a $B$-linear automorphism of $V_l$; we define $\gamma_l\in G(\Q_l)$ as its inverse. Its conjugacy class is well defined, and these elements define a conjugacy class $\gamma\in G(\A_f^p)$. The associated $p$-divisible group $H=A_x[p^\infty]$ then gives an element $\delta\in G(\Q_{p^t})$ well defined up to $\sigma$-conjugation by $G(\Z_{p^t})$. It satisfies the equality $\kappa_{G}(p\delta)=\mu^\sharp$.

Now we get the first form the trace formula. Fix the Haar measures on $G(\Q_p)$, resp. $G(\Q_{p^t})$, that give $G(\Z_p)$, resp. $G(\Z_{p^t})$, volume 1.
\begin{proposition}With the notations above, we have the following equality:
\[tr(\tau\times hf^p|H_\xi)=\sum_{\varepsilon}\tr{vol}(I_\varepsilon(\Q)\setminus I_\varepsilon(\A_f))
O_\gamma(f^p)TO_{\delta\sigma}(\phi_{\tau,h})tr\xi(\varepsilon),\]where $\varepsilon$ runs over the set of conjugacy classes of $I(\Q)$ such that $A_\varepsilon$ is non empty.
\end{proposition}
\begin{proof}
This formula comes from the application of Lefschetz trace formula \cite{Va3} and the description of the set of $\F_{p^t}$-points. The local term of each fixed point in the isogeny class associated to $\varepsilon$ is $\phi_{\tau,h}(\delta)tr\xi(\varepsilon)$. The contribution of fixed points in this isogeny class is given by
\[\tr{vol}(I_\varepsilon(\Q)\setminus I_\varepsilon(\A_f))
O_\gamma(f^p)TO_{\delta\sigma}(\phi_{\tau,h})tr\xi(\varepsilon).\]
For a more detailed proof, see \cite{Sh} proposition 5.1.
\end{proof}

We want to construct a ``Kottwitz triple'' attached to the above $\varepsilon$. Recall the usual definition of a Kottwitz triple for our reductive group $G$ over $\Q$.
\begin{definition}\label{D:Kottwitz}
Let $j\geq 1$. Set $t=j[\kappa_E:\F_p]$ with $\kappa_E$ as the residue field of $E$. A degree-$j$-Kottwitz triple $(\gamma_0;\gamma,\delta)$ consists of
\begin{itemize}
\item a semisimple stable conjugacy class $\gamma_0\in G(\Q)$,
\item a conjugacy class $\gamma\in G(\A_f^p)$ that is stably conjugate to $\gamma_0$,
\item a $\sigma$-conjugacy class $\delta\in G(\Q_{p^t})$ such that $N\delta:=\delta\sigma(\delta)\cdots\sigma^{t-1}(\delta)$ is stably conjugate to $\gamma_0$\end{itemize}
satisfying
\begin{enumerate}
\item $\gamma_0$ is elliptic in $G(\R)$
\item $\kappa_{G_{\Q_p}}(p\delta)=\mu^{\sharp}$ in $X^\ast(Z(\wh{G})^{\Gamma})$, where $\Gamma$ is the absolute Galois group of $\Q_p$.
\end{enumerate}

\end{definition}
The conjugacy class of $N\delta$ is stable under the Galois group $\Gamma$. However, since $G_{\Q_p}$ is not quasi-split, it can happen that this conjugacy class contains no element of $G(\Q_p)$, cf. \cite{Ko6}. This is an obstruction to obtain a $\gamma_0\in G(\Q)$ such that $(\gamma_0;\gamma,\delta)$ forms a Kottwitz triple for $G$. In fact, it is the only obstruction, since we have the following lemma.
\begin{lemma}If the conjugacy class of $N\delta$ contains an element of $G(\Q_p)$, then we can find an element $\gamma_0\in G(\Q)$ such that $(\gamma_0;\gamma, \delta)$ forms a Kottwitz triple for $G$.
\end{lemma}
\begin{proof}
Let $(A,\iota,\lambda)$ be a $c$-polarized virtual abelian variety over $\F_{p^t}$ coming from a $\F_{p^t}$-point of $S_{K^p}$ inside the isogeny class determined by $\varepsilon$. Consider the Frobenius morphism $\pi_A\in End_B(A)$. Then $F(\pi_A)$ is a CM field, which is also the center of the division algebra $End_B(A)$. The Rosati involution $\ast$ on $End_B(A)$ induced by $\lambda$ preserves $F(\pi_A)$. Since by assumption, the conjugacy class of $N\delta$ is stable under the Galois group $\Gamma$, we have an embedding $F(\pi_A)\subset B^{opp}$. Moreover, by the definition of the moduli problem, the two involutions $\ast$ and $\sharp$ are compatible under this embedding. Then we have $\pi_A\in G(\Q)$ under the above embedding. We can take $\gamma_0=\pi_A^{-1}$. The conditions in definition 5.2 for $(\gamma_0;\gamma,\delta)$ being a Kottwitz triple can be verified as in \cite{Ko2} section 14.
\end{proof}
The converse of this lemma is clearly true. In particular, for our group $G$ which is not quasi-split at $p$, the set of Kottwitz triples is not enough for parameterizing all the $\F_{p^t}$-points of our Shimura varieties. Maybe it is possible to define some generalized Kottwitz triples by introducing some suitable inner form of $G$, such that these generalized Kottwitz triples parameterize all the points on the Shimura varieties over finite fields. But we will not pursue this subject here.
Now the key point is the following theorem, which can be viewed as a generalization of the conjecture 5.7 of \cite{Ra} which was proved in the maximal level case by Waldspurger. See also \cite{Ra1} conjecture 10.2 in the more general setting.
\begin{theorem}
If the conjugacy class of $N\delta$ does not contain an element of $G(\Q_p)$, then for any $j\geq 1, \tau\in Frob^jI_E$ and any $h\in C_c^\infty(G(\Z_p))$, we have the twisted orbital integral of the test function $\phi_{\tau,h}$ vanishes
\[TO_{\delta\sigma}(\phi_{\tau,h})=0.\]
In particular, there is a function $f_{\tau,h}\in C_c^\infty(G(\Q_p))$ which has matching (twisted) orbital integrals with $\phi_{\tau,h}\in C_c^\infty(G(\Q_{p^t}))$.
\end{theorem}
\begin{proof}
By the properties of the test function $\phi_{\tau,h}$, we can assume we are in the local quasi-EL case such that $G(\Q_p)=D^\times\times\Q_p^\times$, where $D$ is a central division algebra with invariant $\frac{1}{n}$ over a finite extension $E$ of $\Q_p$. Then the test functions $\phi_{\tau,h}$ are defined for the Shimura varieties which were studied in section 2 at $p$. With the notation of section 2, we have specially simple form: $r=1$ and $G_{D'}=1$.
By proposition 5.1, we have the formula
\[tr(\tau\times hf^p|H_\xi)=\sum_{\varepsilon, A_\varepsilon\neq \emptyset}\tr{vol}(I_\varepsilon(\Q)\setminus I_\varepsilon(\A_f))
O_\gamma(f^p)TO_{\delta\sigma}(\phi_{\tau,h})tr\xi(\varepsilon),\]
where $\varepsilon$ runs over the set of conjugacy classes of $I(\Q)$ such that $A_\varepsilon\neq \emptyset$, with the associated (twisted) conjugacy classes $\gamma$ and $\delta$ as above. Recall that in the above situation $I_{\Q_p}=(Res_{E|\Q_p}GL_n)\times \G_m$. Let $G^\ast=I_{\Q_p}$ be the quasi-split inner form of $G$ over $\Q_p$. Then we have the norm map \[N: \{\sigma-\tr{conjugacy classes in}\, G(\Q_{p^t})\}\longrightarrow \{\tr{conjugacy classes in}\, G^\ast(\Q_p)=GL_n(E)\times\Q_p^\times\}\] from which we can define a transfer $f^\ast_{\tau,h}\in C_c^\infty(G^\ast(\Q_p))$, such that for any conjugacy class $\{\gamma_p\}$ not in the image of $N$ we have
\[O_{\gamma_p}(f^\ast_{\tau,h})=0,\]and for $\{\gamma_p\}=N\{\delta\}$ we have
\[O_{\gamma_p}(f^\ast_{\tau,h})=e(\delta)TO_{\delta\sigma}(\phi_{\tau,h}),\]where $e(\delta)$ is the Kottwitz sign of $G_{\delta\sigma}$. By construction, for $\delta$ coming from $\varepsilon$, we have $N\{\delta\}=\{\gamma_p\}$ where $\{\gamma_p\}$ is the image of $\varepsilon$ in $I(\Q_p)$.
Then we can rewrite the above formula as
\[tr(\tau\times hf^p|H_\xi)=\sum_{\varepsilon}\tr{vol}(I_\varepsilon(\Q)\setminus I_\varepsilon(\A_f))
O_\gamma(f^p)e(\delta)O_{\gamma_p}(f^\ast_{\tau,h})tr\xi(\varepsilon).\]Note in particular by definition of $f^\ast_{\tau,h}$, in the above sum $\varepsilon$ runs over the set of all conjugacy classes of $I(\Q)$.
We identify the center $Z\subset G$ as a subgroup of $I$ and $I_\varepsilon$. Recall the Tamagawa number of $I_\varepsilon$ is defined by
\[\tau(I_\varepsilon)=\tr{vol}(I_\varepsilon(\Q)A_G(\R)^0\setminus I_\varepsilon(\A)),\]where $A_G$ is the split component of $Z$ and $A_G(\R)^0$ is the connected component containing the identity element. So the factor $\tr{vol}(I_\varepsilon(\Q)\setminus I_\varepsilon(\A_f))$ equals
 \[\tau(I_\varepsilon)\tr{vol}(A_G(\R)^0\setminus I_\varepsilon (\R))^{-1}.\]Let $f_\infty$ be a pseudo-coefficient on $I(\R)$ for $\check{\xi}$. Denote the image of $\varepsilon$ in $I(\R)$ by $\gamma_\infty$. Then as p. 659 of \cite{Ko3} (recall that the group $I(\R)$ is compact modulo center) we have \[O_{\gamma_\infty}(f_\infty)=e(\gamma_\infty)\tr{vol}(A_{G}(\R)^0\setminus I_\varepsilon(\R))^{-1}tr\xi(\varepsilon),\]where $e(\gamma_\infty)$ is the Kottwitz sign of $I_{\R}$. As in \cite{Ra2} p. 690 or \cite{Re} section 10, we have $e(\delta)=e(\gamma_\infty)$. Therefore, we get
\[tr(\tau\times hf^p|H_\xi)=\sum_{\varepsilon}\tau(I_\varepsilon)O_{\gamma}(f^p)O_{\gamma_p}(f^{\ast}_{\tau,h})
O_{\gamma_\infty}(f_\infty).\]
Now we apply the simple trace formula for the group $I$ and the function $f=f^pf_{\tau,h}^{\ast}f_\infty$. The above equals
\[\sum_{\pi}m(\pi)tr\pi(f),\]
where $\pi$ runs over the automorphic representations of $I(\A)$ whose central character is the inverse of that of $\xi$ on $A_{G}(\R)^0$.

On the other hand, we have known by theorem 3.4
\[H_\xi=\sum_{\pi_f}a(\pi_f)\pi_f\otimes(r_{-\mu}\circ\varphi_{\pi_p}|_{W_E})|-|^{(1-n)/2}\]
as virtual $G(\A_f)\times W_E$-representations. Now by the comparison of this formula with the above trace formula, we can conclude that for any irreducible smooth representation $\pi_p=\pi_p^0\otimes\chi_p$ of $G^\ast(\Q_p)=GL_n(E)\times\Q_p^\times$, if $\pi_p^0$ is not a discrete series, i.e. it does not come from an irreducible smooth representation of $D^\times$, then
\[tr\pi_p(f_{\tau,h}^\ast)=0.\]
Indeed, for any such $\pi_p$, suppose first that we can find an automorphic representation $\pi$ of $I$ with the $p$-component as $\pi_p$ and $m(\pi)\neq 0$. Take compact open subgroups $K^p\subset I(\A_f^p)=G(\A_f^p), K_p\subset I(\Q_p)$ such that $\pi_f^p$ has $K^p$-invariants and $\pi_p$ has $K_p$-invariants. Let $K=K^pK_p$. Then there are only finitely finitely irreducible admissible representations $\pi_f'$ of $I(\A_f)$ such that $\pi_f'$ occurs as the finite adelic component of an automorphic representation with central character the inverse of $\xi$ on $A_{G}(\R)^0$, $\pi_f'$ has $K$-invariants, and $m(\pi')\neq 0$. One can then find a function $f^p\in C_c^\infty(I(\A_f^p))$ biinvariant under $K^p$ with $tr(f^p|\pi_f^p)=1$ and such that whenever $\pi_f'$ is an irreducible admissible representation of $I(\A_f)$ with $tr(\pi_f')^p(f^p)\neq 0$, with $(\pi_f')^p$ has invariants under $K^p$, then $(\pi_f')^p\simeq \pi_f^p$ (which implies $\pi_f'=\pi_f$ by lemma 3 of \cite{H1}. In fact the situation of loc. cit. is under more restrictive hypotheses, but the same proof applies.). Take such an $f^p$. Recall that by our choice of $f_\infty$ we have $tr\pi_\infty(f_\infty)=1$. Then the right hand of the trace formula has only one term
\[m(\pi)tr\pi_p(f^\ast_{\tau,h}),\]which has to be zero according to the description of $H_\xi$. 

For the general case, we use the facts that the set of local $p$-components of automorphic representations of $I$ is Zariski dense in the Bernstein variety of $G^\ast(\Q_p)$, see for example proposition 3.1 of \cite{Shin}. More precisely, let \[\mathfrak{z}_2=\coprod_{(L,D)\in\mathfrak{S}(GL_n(E))}V(L,D)\] be the Bernstein variety of $GL_n(E)$ as introduced in \cite{DKV} A.4. Here we use the notation of 2.1 of \cite{Shin}. This is the discrete series variant of the usual Bernstein variety $\mathfrak{z}$. Let $Irr(GL_n(E))$ be the set of isomorphism classes of irreducible smooth representations of $GL_n(E)$. Then thanks to the Bernstein-Zelevinsky classification for $GL_n(E)$, we get a map with finite fibers
\[r: Irr(GL_n(E))\ra \mathfrak{z}_2,\]such that its composition with the finite map $\mathfrak{z}_2\ra \mathfrak{z}$ (cf. remark 2.2 of \cite{Shin}, and see also the example of A.4.e in p. 64 of \cite{DKV} for an explicit description of this finite map) is the usual map $Irr(GL_n(E))\ra   \mathfrak{z}$ sending a representation to its supercuspidal support. By the classification of  discrete series in \cite{Z}, the set of isomorphism classes of discrete series of $GL_n(E)$ is exactly the inverse image under $r$ of the components $V(L,D)$ with $L=G$ (see also example 2.3 of \cite{Shin}). Take any $(L,D)\in \mathfrak{S}(GL_n(E))$ with $L\neq G$. Proposition 3.1 of \cite{Shin} says that the set of local $p$-components of automorphic representations of $I$ (which are non discrete series) \[Y=\{x\in V(L,D)|\,\exists\, \Pi\subset \mathcal{A}(I)_\xi,\, x=r(\Pi_{p,0})\}\] is Zariski dense in $V(L,D)$. (Note that by our definitions of $\mathcal{A}(I)_\xi$ in section 3 and of $f_\infty$, for $\Pi\subset \mathcal{A}(I)_\xi$, we have $tr\Pi_\infty(f_\infty)\neq 0$.)
By proposition 2.2 of \cite{Shin}, the function $\pi_p\longmapsto tr\pi_p(f_{\tau,h}^\ast)$ is regular on the Bernstein variety $\mathfrak{z}_2$. Therefore, $tr\pi_p(f_{\tau,h}^\ast)=0$ for any non discrete series representation $\pi_p$.

This implies if $\gamma_p$ does not come from an element of $G(\Q_p)$, i.e. the conjugacy class of $N\delta$ does not include an element of $G(\Q_p)$, then
\[O_{\gamma_p}(f^{\ast}_{\tau,h})=0,\]cf. \cite{Bad} lemma 3.3. This means that $f_{\tau,h}^{\ast}$ comes from a function $f_{\tau,h}\in C_c^\infty(G(\Q_p))$, such that \[tr\pi_p(f_{\tau,h})=e(\delta)tr\pi_p^\ast(f_{\tau,h}^{\ast}),\]where $\pi_p\longmapsto \pi_p^\ast$ is the Jacquet-Langlands correspondence between the set of smooth irreducible representations of $G(\Q_p)$ and $G^\ast(\Q_p)$. The functions $f_{\tau,h}^{\ast}$ and $f_{\tau,h}$ have matching orbital integrals, hence $f_{\tau,h}$ and $\phi_{\tau,h}$ have matching (twisted) orbital integrals.

\end{proof}

We have the following theorem. It says that the fixed points which have non trivial contribution to the trace formula can be parameterized by Kottwitz triples, and their contribution have the usual description as in the quasi-split case of \cite{Sch3}.
\begin{theorem}
Let $f^p\in C_c^\infty(G(\A_f^p)), h\in C_c^\infty(G(\Z_p))$ and $\tau\in Frob^jI_E\subset W_E$, then
\[tr(\tau\times hf^p|H_\xi)=\sum_{(\gamma_0;\gamma,\delta)}c(\gamma_0;\gamma,\delta)O_\gamma(f^p)TO_{\delta\sigma}(\phi_{\tau,h})
tr\xi(\gamma_0),\]
where the sum runs over degree-$j$-Kottwitz triples, and $c(\gamma_0;\gamma,\delta)$ is a volume factor defined as in \cite{Ko2} p. 441. The Haar measures on $G(\Q_p)$ resp. $G(\Q_{p^t})$ are normalized by giving $G(\Z_p)$ resp. $G(\Z_{p^t})$ volume 1.
\end{theorem}
\begin{proof}
By proposition 5.1 we have the formula
\[tr(\tau\times hf^p|H_\xi)=\sum_{\varepsilon}\tr{vol}(I_\varepsilon(\Q)\setminus I_\varepsilon(\A_f))
O_\gamma(f^p)TO_{\delta\sigma}(\phi_{\tau,h})
tr\xi(\varepsilon),\]where $\varepsilon$ runs over the set of conjugacy classes of $I(\Q)$ such that $A_\varepsilon$ is non empty.
By the above vanishing theorem, only those $\varepsilon$ such that the conjugacy class of $N\delta$ for the associated $\delta$ have non trivial contribution to the this formula, in which case we can find a $\gamma_0\in G(\Q)$ such that $(\gamma_0;\gamma,\delta)$ forms a Kottwitz triple for $G$. Moreover, in this case the associated group $I(\gamma_0;\gamma,\delta)$ to $(\gamma_0;\gamma,\delta)$ can be taken as $I_\varepsilon$. Fix such a Kottwitz triple $(\gamma_0;\gamma,\delta)$. The number of conjugacy classes $\varepsilon$ which gives this triple equals to $|ker(ker^1(\Q,I(\gamma_0;\gamma,\delta))\ra ker^1(\Q,G))|$, cf. \cite{Mi} proposition 6.11. Clearly the traces of $\varepsilon$ and $\gamma_0$ are the same for the $\ov{\Q}_l$-representation $\xi$. Then one can rewrite the above formula in the form as in the theorem.

\end{proof}

\section{A character identity}
We continue the study of the cohomology group $H_\xi$. Let the notations be as in last section.
Since the global reductive group $G$ has trivial endoscopic groups, by the procedure of pseudostabilization, we get the following corollary.
\begin{corollary}For $\tau\in Frob^jI_E\subset W_E, \, h_1\in C_c^\infty(O_{D_1}^\times),\dots,h_r\in C_c^\infty(O_{D_r}^\times),h'\in C_c^\infty(O_{D'}^\times)$, $h_0\in C_c^\infty(\Z_p^\times)$, with all these functions take values in $\Q$ and $h'$ invariant under conjugation,
we have the following equality
\[Ntr(\tau\times h_1\times \cdots\times h_r\times h'\times h_{0}\times f^p|H_\xi)=tr(f_{\tau,h_1}\times\cdots\times f_{\tau,h_r}\times h'\times h_0'\times f^p|H_\xi),\]where $N=dim\, r_{-\mu}$, $h_0'\in C_c^\infty(\Q_p^\times)$ has support in $p^{-t}\Z_p^\times$ and is defined by $h_0'(x)=h_0(p^tx)$ for all $x\in p^{-t}\Z_p^\times$
\end{corollary}
\begin{proof}
With the expression in the last theorem, we can apply Kottwitz's argument of pseudostabilization in \cite{Ko3} (see also the proof of corollary 9.4 in \cite{Sch2}). We just sketch the most important steps in the calculation. First, note the functions $\phi_{\tau,h}=\phi_{\tau,h_1}\times\cdots\phi_{\tau,h_r}\times \phi_{h'}\times h_0$ and $f_{\tau,h_1}\times\cdots\times f_{\tau,h_r}\times h'\times h_0'$ have matching (twisted) orbital integrals. As in \cite{Ko3} p. 657 Kottwitz introduced a function $f_\infty$ on $G(\R)$ depending on $\xi$. Consider the function
\[f=f_{\tau,h_1}\times\cdots\times f_{\tau,h_r}\times h'\times h_0'\times f^p\times f_\infty\]
be the function on $G(\A)$. Then arguing as p.661-663 of loc. cit. we get the left hand side equals
\[N\tau(G)\sum_{\gamma_0}SO_{\gamma_0}(f),\]
where $\tau(G)$ is the Tamagawa number of $G$, $\gamma_0$ runs through the stable conjugacy classes in $G(\Q)$ and $SO_{\gamma_0}(f)$ is the stable orbital integral. By the Arthur-Selberg trace formula for $f$ we get it equals
\[N\sum_{\pi}m(\pi)tr(f|\pi),\]where $\pi$ runs through automorphic representations for $G$ with central character $\check{\xi}$. By lemma 4.2 of \cite{Ko3} this can be rewritten as
\[\sum_{\pi_f}tr(f_{\tau,h_1}\times\cdots\times f_{\tau,h_r}\times h'\times h_0'\times f^p|\pi_f)
\sum_{\pi_\infty}m(\pi_f\otimes\pi_\infty)ep(\pi_\infty\otimes\xi),\]where $ep$ denotes the Euler-Poincare characteristic as in loc. cit.. Now Matsushima's formula shows that this equals
\[tr(f_{\tau,h_1}\times\cdots\times f_{\tau,h_r}\times h'\times h_0'\times f^p|H_\xi).\]

\end{proof}

Consider the case $r=1$. Then the above corollary combined with theorem 3.4 gives us the following character identity.
\begin{proposition}For
$\rho\in Irr(D^\times)$ with $L$-parameter $\varphi_\rho$, $h\in C_c^\infty(O_D^\times)$ with values in $\Q$ and the conjugacy class $\{\mu\}$ of cocharacters $\mu$ defined in section 4, we have the identity
\[tr(f_{\tau,h}|\rho)=tr(\tau|(r_{-\mu}\circ\varphi_\rho|_{W_E})|-|^{(1-n)/2})tr(h|\rho).\]
\end{proposition}
\begin{proof}
As in the proof of theorem 8.1 in \cite{SS}, it is equivalent to prove the equality for the corresponding quasi-EL case.
With the notation of section 3, we have $G_{\Q_p}=D^\times\times\Q_p^\times$.
By theorem 3.4 we know
\[H_\xi=\sum_{\pi_f}a(\pi_f)\pi_f\otimes(r_{-\mu}\circ\varphi_{\pi_p}|_{W_E})|-|^{(1-n)/2}.\]
By lemma 1 of \cite{H1},
we can globalize $\rho$ to an irreducible admissible representation $\pi_f$ of $G(\A_f)$ such that there is an algebraic representation $\xi$ of $G$ with the $\pi_f$-isotypic component of $H_\xi$ \[H(\pi_f):=Hom_{G(\A_f)}(\pi_f, H_\xi)\neq 0.\]Take some integer $m\geq 1$ such that $h\in C_c^\infty(G(\Z_p))$ is bi-$K^m_p$-invariant. Also, take a compact open compact subgroup $K^p\subset G(\A_f^p)$ such that $\pi_f^p$ has $K^p$-invariants. Let $K=K^m_pK^p$, then since $H_\xi^K=\sum(-1)^iH^i(Sh_{K,\ov{\Q}_p},\mathcal{L}_\xi)$ is finite dimensional, there are only finitely many irreducible admissible representations $\pi_f'$ with invariants under $K$ and $H(\pi_f')\neq 0$. There exists a function $f^p\in C_c^\infty(G(\A_f^p))$ bi-$K^p$-invariant with $tr(f^p|\pi_f^p)=1$ and such that whenever $\pi_f'$ is an irreducible representation of $G(\A_f)$ with $H(\pi_f')\neq 0$, with invariants under $K$, and $tr(f^p|(\pi_f')^p)\neq 0$, then $\pi_f\simeq \pi_f'$. Note the $N$ of corollary 6.1 equals $n$ under our assumption here.
Now we compute the trace of the function $f_{\tau,h}\times f^p$ on $H_\xi$:
\[tr(f_{\tau,h}\times  f^p|H_\xi)=na(\pi_f)tr(f_{\tau,h}|\pi_p).\]
By the above corollary
\[na(\pi_f)tr(\tau|r_{-\mu}\circ \varphi_{\pi_p}|_{W_E}|-|^{(1-n)/2})tr(h|\pi_p)=na(\pi_f)tr(f_{\tau,h}|\pi_p).\]
We can deduce the desired identity.

\end{proof}

\section{The cohomology of Shimura varieties II}

Now consider the general case that $1\leq r\leq s$ are arbitrary integers. The cocharacter $\mu: \G_m\lra G_{\ov{\Q}_p}$ has a decomposition $\mu=(\mu_1,\dots,\mu_r,\mu',\mu_0)$, which induces a decomposition of $r_{-\mu}=(\otimes_{i=1}^rr_{-\mu_i})\otimes r_{-\mu'}\otimes r_{-\mu_0}$ when restricting on $\wh{G}$. In fact one sees easily that $r_{-\mu'}$ is the trivial representation. $E$ is the composition of local reflex fields $F_{\vp_i}$ for each $\mu_i$ with $1\leq i\leq r$. Any smooth irreducible representation $\pi$ of $G(\Q_p)=\prod_{i=1}^rD_i^\times\times (D')^\times\times\Q_p^\times$ has a tensor product decomposition $\pi=\otimes_{i=1}^r\pi_i\otimes \pi'\otimes\chi$. Accordingly we have $L$-parameters $\varphi_\pi,\varphi_{\pi_1},\dots,\varphi_{\pi_r},\varphi_{\pi'},\varphi_\chi$. Then the properties of the test functions lead to the following character identity in the general case.
\begin{proposition}
For any irreducible representation $\pi$ of $G(\Q_p)$ with $L$-parameter $\varphi_\pi$, let $h\in C_c^\infty(G(\Z_p))$ have the form as $h=h_1\times\cdots\times h_r\times h'\times h_0$ with $h_i\in C_c^\infty(O_{D_i}^\times)$ for $1\leq i\leq r$, $h'\in C_c^\infty(O_{D'}^\times)$, and $h_0\in C_c^\infty(\Z_p^\times)$, we have the identity
\[tr(f_{\tau,h}|\pi)=tr(\tau|(r_{-\mu}\circ \varphi_\pi|_{W_E})|-|^{r(1-n)/2})tr(h|\pi).\]
\end{proposition}
\begin{proof}
This comes from the decompositions of $f_{\tau,h},\, (r_{-\mu}\circ\varphi_\pi|_{W_E})|-|^{r(1-n)/2}$, $h$ and the corresponding identities for each factor $f_{\tau,h_i},\,1\leq i\leq r$, and for $h_0'$.
\end{proof}

\begin{theorem}
We have the identity
\[H_\xi=\sum_{\pi_f}a(\pi_f)\pi_f\otimes(r_{-\mu}\circ\varphi_{\pi_p}|_{W_E})|-|^{r(1-n)/2}\]
as virtual $G(\Z_p)\times G(\A_f^p)\times W_E$-representations. The notations in this identity are the same as those in theorem 3.4.

\end{theorem}
\begin{proof}
As in the proof of corollary 6.1, we can take the pseudostabilization of the formula in theorem 5.4 to get the equality
\[tr(\tau\times h\times f^p|H_\xi)=N^{-1}tr(f_{\tau,h}f^p|H_\xi).\]On the other hand, Matsushima's formula implies
\[H_\xi=N\sum_{\pi_f}a(\pi_f)\pi_f.\]Put this into the above formula and take account the equality of proposition 7.1, we get
\[ tr(\tau\times h\times f^p|H_\xi)=\sum_{\pi_f}a(\pi_f)tr(\tau|(r_{-\mu}\circ\varphi_{\pi_p})|-|^{r(1-n)/2})tr(hf^p|\pi_f).\]
This gives the desired identity.
\end{proof}
As in the proof of the theorem, the crucial point is that we have the character identity of proposition 6.2, which in turn needs first theorem 5.4 to hold true to have the definitions of the functions $f_{\tau,h}$. This identity plus that in the theorem 8.1 of \cite{SS} can be used to prove new cases for the description of cohomology of Shimura varieties. For example, the compact unitary Shimura varieties with trivial endoscopy, such that the local reductive groups at $p$ are products of (Weil restrictions of) $D^\times$ and $GL_n$ where $D$ is a central division algebra over a finite extension of $\Q_p$ with invariant $\frac{1}{n}$. Moreover, to apply the results in this paper, we have to require the local cocharacters for the factors $D^\times$ are the same as those studied here. The case of $n=2$ in \cite{Ko3} but for arbitrary $p$ (with the above requirement on local cocharacters at ramified places) will be a typical example. We will treat this case and the related quaternionic Shimura varieties in \cite{Sh}.

From this theorem we get the following corollary concerning the local semisimple Hasse-Weil zeta functions of our Shimura varieties. For the definition of local semisimple Hasse-Weil zeta functions and local semisimple automorphic $L$ functions, see \cite{Ra}. Let $\wt{E}$ be the global reflex field, and $\nu$ be a place of $\wt{E}$ above $p$ such that $E=\wt{E}_\nu$.
\begin{corollary}
In the situation of the theorem, let $K\subset G(\A_f)$ be any sufficiently small open compact subgroup. Then the semisimple local Hasse-Weil zeta function of $Sh_K$ at the place $\nu$ of $\wt{E}$ is given by
\[\zeta_{\nu}^{ss}(Sh_K,s)=\prod_{\pi_f}L^{ss}(s-r(n-1)/2,\pi_p,r_{-\mu})^{a(\pi_f)dim\pi_f^K}.\]
\end{corollary}
\begin{proof}We can assume that $K$ has the form as $K^pK_p\subset G(\A_f^p)\times G(\Z_p)$. Then the corollary follows from the previous theorem and the definitions.
\end{proof}
In the case $r=1$ Dat has proved the Weight-Monodromy conjecture for these Shimura varieties, cf. \cite{Dat} 5.2. In fact, there the varieties involved are of the form $\M_{Dr,m}/\Gamma$ where $\Gamma\subset GL_n(E)$ is some torsion free discrete subgroup. Here our Shimura varieties have the same connected components as $\M_{Dr,m}/\Gamma$ for suitable $\Gamma$. By \cite{Ra} section 2 one can recover the classical Hasse-Weil zeta function.
\begin{corollary}
Let $r=1$ and $K\subset G(\A_f)$ be any sufficiently small open compact subgroup in the situation of the theorem. Then the local Hasse-Weil zeta function of $Sh_K$ at the place $\nu$ of $\wt{E}$ is given by
\[\zeta_{\nu}(Sh_K,s)=\prod_{\pi_f}L(s-r(n-1)/2,\pi_p,r_{-\mu})^{a(\pi_f)dim\pi_f^K}.\]
\end{corollary}
Finally we remark that, assuming the results of \cite{Mie}, in the above corollary there will be no restriction for the integer $r$, i.e. it can be an arbitrary positive integer, see the paragraph under corollary 1.3 in the introduction.

\end{document}